\documentclass[aos]{imsart}
\RequirePackage{amsthm,amsfonts,amssymb}
\RequirePackage[numbers,sort]{natbib}
\RequirePackage[colorlinks,citecolor=blue,urlcolor=blue]{hyperref}
\RequirePackage{graphicx}
\usepackage[reqno]{amsmath}

\usepackage{enumerate}
\usepackage{dsfont}
\usepackage[usenames,dvipsnames]{color}
\usepackage{mathrsfs}
\usepackage{placeins}
\usepackage{graphicx}
\usepackage{tikz}
\usetikzlibrary{arrows}
\usepackage{url}
\usepackage{comment}
\usepackage{multicol}
\usepackage{cleveref}
\usepackage{stmaryrd}
\usepackage{calc}
\usepackage{xfrac}
\usepackage{tikz}
\tikzstyle{vertex}=[circle, draw, inner sep=0pt, minimum size=6pt]

\newcommand{\revision}[1]{{\color{black} #1}}

\usepackage{natbib}

\renewcommand{\thetable}{\arabic{section}.\arabic{table}}
\numberwithin{equation}{section}

\theoremstyle{plain} 
\newtheorem{theorem}{Theorem}[section]
\newtheorem{lemma}[theorem]{Lemma}
\newtheorem{proposition}[theorem]{Proposition}
\newtheorem{corollary}[theorem]{Corollary}

\theoremstyle{remark}
\newtheorem{definition}[theorem]{Definition}

\newtheorem{example}[theorem]{Example}

\newtheorem{remark}[theorem]{Remark}

\def\tablename{\footnotesize Table}

\newcommand{\bthe}{\begin{theorem}}
\newcommand{\ethe}{\end{theorem}}

\newcommand{\ben}{\begin{enumerate}}
\newcommand{\een}{\end{enumerate}}

\newcommand{\bit}{\begin{itemize}}
\newcommand{\eit}{\end{itemize}}

\newcommand{\beq}{\begin{equation}}
\newcommand{\eeq}{\end{equation}}

\newcommand{\ble}{\begin{lemma}}
\newcommand{\ele}{\end{lemma}}

\newcommand{\bde}{\begin{definition}\rm}
\newcommand{\ede}{\halmos\end{definition}}

\newcommand{\bco}{\begin{corollary}}
\newcommand{\eco}{\end{corollary}}

\newcommand{\bpr}{\begin{proposition}}
\newcommand{\epr}{\end{proposition}}

\newcommand{\brem}{\begin{remark}\rm}
\newcommand{\erem}{\end{remark}}

\newcommand{\bproof}{\begin{proof}}
\newcommand{\eproof}{\end{proof}}

\newcommand{\bexam}{\begin{example}\rm}
\newcommand{\eexam}{\end{example}}

\newcommand{\bfi}{\begin{fig}}
\newcommand{\efi}{\end{fig}}

\newcommand{\btab}{\begin{tab}}
\newcommand{\etab}{\end{tab}}

\newcommand{\beao}{\begin{eqnarray*}}
\newcommand{\eeao}{\end{eqnarray*}\noindent}

\newcommand{\balo}{\begin{align*}}
\newcommand{\ealo}{\end{align*}}

\newcommand{\balm}{\begin{align}}
\newcommand{\ealm}{\end{align}\noindent}

\newcommand{\beam}{\begin{eqnarray}}
\newcommand{\eeam}{\end{eqnarray}\noindent}

\newcommand{\barr}{\begin{array}}
\newcommand{\earr}{\end{array}}

\newcommand{\E}{\mathbb{E}}

\newcommand{\N}{\mathbb{N}}
\renewcommand\P{\mathbb{P}}
\newcommand{\R}{\mathbb{R}}
\newcommand{\MO}{\mathrm{MO}}

\newcommand{\MDA}{\mathrm{MDA}}

\def\II{\boldsymbol{I}}

\def\bzero{\boldsymbol 0}
\def\bone{\boldsymbol 1}

\def\bM{\boldsymbol M}

\def\bX{\boldsymbol X}

\def\bZ{\boldsymbol Z}

\def\be{\boldsymbol e}

\def\bv{\boldsymbol v}

\def\bx{\boldsymbol x}
\def\by{\boldsymbol y}
\def\bz{\boldsymbol z}

\def\MRV{\mathcal{MRV}}

\def\RV{\mathcal{RV}}

\newcommand{\VaR}{\mathrm{VaR}}

\newcommand{\CoVaR}{\mathrm{CoVaR}}
\newcommand{\CTP}{\mathrm{CTP}}

\newcommand{\leftinv}{{\leftarrow}}

\newcommand{\ov}{\overline}

\newcommand{\vague}{\stackrel{\lower0.2ex\hbox{$\scriptscriptstyle
                    \it{v} $}}{\rightarrow}}
\newcommand{\weak}{\stackrel{\lower0.2ex\hbox{$\scriptscriptstyle
                    \it{w} $}}{\rightarrow}}
\newcommand{\what}{\stackrel{\lower0.2ex\hbox{$\scriptscriptstyle
                    \it{\hat{w}} $}}{\rightarrow}}
\newcommand{\eqdis}{\stackrel{\lower0.2ex\hbox{$\scriptscriptstyle
                    \mathrm{d}$}}{=}}
\newcommand{\distr}{\stackrel{\lower0.2ex\hbox{$\scriptscriptstyle
                    \it{d} $}}{\rightarrow}}
\allowdisplaybreaks 

\begin{document}

\begin{frontmatter}

\title{Asymptotic independence in higher dimensions \vspace*{0.2cm}\\  and its implications on risk management \vspace*{0.8cm}}

\runtitle{Asymptotic independence}

\begin{aug}
  \author[A]{\fnms{Bikramjit} \snm{Das}\ead[label=e1]{bikram@sutd.edu.sg}\orcid{0000-0002-6172-8228}}
    \and
   \author[B]{\fnms{Vicky} \snm{Fasen-Hartmann}\ead[label=e2]{vicky.fasen@kit.edu}\orcid{0000-0002-5758-1999}}
 \address[A]{Engineering Systems and Design, Singapore University of Technology and Design  \printead[presep={,\ }]{e1}}

   \address[B]{Institute of Stochastics, Karlsruhe Institute of Technology\printead[presep={,\ }]{e2}}


  \runauthor{B. Das and  V. Fasen-Hartmann}
\end{aug}

\begin{abstract}
   In the study of extremes, the presence of \emph{asymptotic independence} signifies that extreme events across multiple variables are probably less likely to occur together. Although 
well-understood in a bivariate context, the concept remains relatively unexplored when addressing the nuances of the joint occurrence of extremes in higher dimensions. In this paper, we propose a notion of \emph{mutual asymptotic independence} to capture the behavior of joint extremes in dimensions larger than two and contrast it with the classical notion of \emph{(pairwise) asymptotic independence}. Additionally, we define \emph{$k$-wise asymptotic independence}, which captures the tail dependence between pairwise and mutual asymptotic independence. The concepts are compared using examples of Archimedean, Gaussian and Marshall-Olkin copulas, among others.
\revision{Finally}, we discuss the implications of these new notions of asymptotic independence on assessing the risk of complex systems under distributional ambiguity.


\end{abstract}



\end{frontmatter}

\section{Introduction}\label{sec:intro}
In many multivariate models, we observe that the likelihood of the joint occurrence of extreme values in two or more variables is negligible in comparison to the occurrence of an extreme value in one variable. In this context, the notion of \emph{asymptotic independence} looms large
in the study of joint extreme values in probability distributions, although mostly restricted to the bivariate set-up.  A random vector $(Z_1,Z_2)\in \R^2$ with identically distributed marginals  is \emph{asymptotically (right-tail/upper-tail) independent} if
\begin{align}\label{eq:bvasyind}
   \P(Z_1>t, Z_2>t) = o(\P(Z_1>t)), \quad t\to \infty,
\end{align}
or equivalently $\P(Z_1>t|Z_2>t) \to 0$ as $t\to\infty$. For the rest of this paper, we focus only on extremes in the non-negative quadrant and drop the terms \emph{right/upper-tail} for convenience.

Often called \emph{Sibuya's condition}, \eqref{eq:bvasyind} was exhibited by \cite{sibuya:1960} for bivariate normal random vectors with any correlation $\rho<1$. Such a limit behavior has also been found to hold for bivariate distributions with arbitrary choice of marginals possessing a variety of dependence structures, including Frank copula, Ali-Mikhail-Haq copula, Gaussian copula, Farlie-Gumbel-Morgenstern copula, and more; see \cite{ledford:tawn:1996,ledford:tawn:1998, coles:heffernan:tawn:1999, heffernan:2000}.
It is widely believed that the presence of asymptotic independence hinders the computation of joint tail probabilities, and has led to a variety of techniques for modeling and estimating rare tail probabilities when such a property is present; see \cite{ledford:tawn:1996,resnick:2002,ramos:ledford:2009,lehtomaa:resnick:2020,das:embrechts:fasen:2013,das:fasen:kluppelberg:2022}.  Nevertheless, for random vectors in dimensions higher than two, limited expositions are available, and multivariate asymptotic independence is often understood to be \eqref{eq:bvasyind} holding for all pairs of variables, which we call \emph{pairwise asymptotic independence}. \revision{Such a notion of multivariate asymptotic independence possibly has its origins in the study of extremes.  For instance,  in \citet[Chapter 5.5]{resnickbook:2008}, a multivariate distribution is called \emph{multivariate asymptotically independent} if it is in the maximum domain of attraction of a multivariate extreme value distribution with independent marginals. Additionally, it is shown that such a characterization of ``multivariate asymptotic independence'' is equivalent to having ``pairwise asymptotic independence'' (assuming identical marginals in the maximum domain of attraction of a univariate extreme value distribution); see \citet[Proposition 5.27]{resnickbook:2008}, \citet[Corollary 5.3.1]{galambos:1978}. In this paper, we show that asymptotic tail independence often may have a much subtler form that goes beyond pairwise asymptotic independence (multivariate asymptotic independence).}

Asymptotic independence for bivariate joint tails is also popularly understood using the coefficient of tail dependence  $\eta$ defined in \citet{ledford:tawn:1996}.  If $Z_1, Z_2$ are identically unit Fr\'echet distributed with distribution function $F(z)=\text{e}^{-1/z},$ $z>0$, and 
  \begin{align*}
    \P(Z_1>t,Z_2>t) = t^{-1/\eta}\ell(t), \quad \quad t\to \infty,  
  \end{align*}
  where $1/2\le \eta <1$ and $\ell$ is slowly varying at infinity (i.e., $\ell(tz)/\ell(z) \to 1$ as $t\to\infty$, $\forall\,z>0$), 
  then $\eta$ represents this \textit{coefficient of tail dependence}. According to \citet{ledford:tawn:1996}, (i) $\eta=1/2$ and $\ell(t)\geq 1$ signifies \textit{near independence}, (ii) $\eta=1$ and $\ell(t)\nrightarrow 0$ as $t\to\infty$ signifies \textit{upper tail dependence}, and finally, (iii) either $1/2<\eta<1$, or $\eta =1$ and $\ell(t)\to 0$ as $t\to \infty$ signifies \textit{positive association}.

 


 The coefficient of tail dependence is a 2-dimensional concept and has been extended to $d$-dimensions as \textit{upper tail order} by \citet{hua:joe:2011} through the survival copula.
Prior to further discussions, we recall the notions of copula and survival copula. 

A copula $C:[0,1]^d \to [0,1]$ is a multivariate distribution function with identical uniform $[0,1]$ marginals. From 
Sklar's Theorem \cite{sklar:1959,nelsen:2006,Durante:2016} we know that for any $d$-dimensional random vector $\bZ=(Z_1, \ldots, Z_d)$ with distribution function $F$ and marginal distributions $F_1,\ldots, F_d$ there exists a copula $C:[0,1]^d \to [0,1]$  such that 
\begin{align*}
    F(z_1,\ldots,z_d)
    = C(F_1(z_1), \ldots, F_d(z_d)) \quad 
\end{align*}
 for  $(z_1,\ldots,z_d)\in \R^d$, and if the marginals are continuous, the copula is uniquely given by
\begin{eqnarray*}
    C(u_1,\ldots,u_d)=F(F_1^{\leftarrow}(u_1),\ldots,F_d^{\leftarrow}(u_d))
\end{eqnarray*}
 for $0< u_1,\ldots, u_d < 1$, where $$F_j^{\leftarrow}(u_j):=\inf\{z\in\R:F_j(z)\geq u_j \}$$ is the generalized inverse of $F_j$  for  $j=1,\ldots,d$.  In this paper, we are particularly concerned with the probability of joint extremes where the survival copula  $\widehat{C}:[0,1]^d \to [0,1]$, which is also a copula,  plays an important role; \revision{see \citet[Chapter 1]{Durante:2016}, \citet[Section 5.1.5]{mcneil:frey:embrechts:2015}, \citet[Section 2.6]{nelsen:2006}}. The survival copula $\widehat{C}$ satisfies  
\begin{align*}
   \P(Z_1 > z_1,\ldots, Z_d > z_d) & = \widehat{C}(\overline{F}_1(z_1), \ldots, \overline{F}_d(z_d)),
\end{align*}
for $(z_1,\ldots,z_d)\in \R^d$, where $\overline{F}_j=1-F_j$ is the tail function of $F_j$ for $j=1,\ldots,d$. Of course, the survival copula and the copula are directly related through 
\begin{eqnarray*}
    \widehat C(u_1,\ldots,u_d)=1+\sum_{S\subseteq \{1,\ldots,d\}\atop S\neq \emptyset}
        (-1)^{|S|}C_S(1-u_j:j\in S)
\end{eqnarray*}
for $0\leq u_1,\ldots, u_d \leq 1$,
where $|S|$ is the cardinality of the set $S$ and $C_S$ is the appropriate $|S|$-dimensional marginal copula of $C$. In dimension $d=2$, this reduces to
\begin{eqnarray*}
    \widehat C(u_1,u_2)=u_1+u_2-1+C(1-u_1,1-u_2),
\end{eqnarray*}
for $0\leq u_1,u_2\leq 1.$

Returning back to notions of \emph{tail dependence}, if a $d$-dimensional survival copula $\widehat{C}$ satisfies
  \begin{align}\label{eq:uppertailorder}
       \widehat{C}(u,\ldots,u) = u^{\kappa}\ell(u), \quad 0\leq u\leq 1,
  \end{align}
  for some slowly varying function $\ell$ at 0 
  and some constant $\kappa>0$, then $\kappa$ is called the \textit{upper tail order}. Here, (i) the case $\kappa=d$ signifies \textit{near (asymptotic) independence} (for $d=2$, we have $\kappa=1/\eta$), (ii) the case $\kappa=1$ and  $\ell(u)\nrightarrow 0$ as $u\downarrow0$ signifies \textit{(asymptotic) upper tail dependence}, and, (iii) the case where $1<\kappa<d$ is called  \textit{upper intermediate tail dependence} in  \citet{hua:joe:2011}. From the definition of \emph{tail order}, we can see that for $d=2$, the survival copulas in both the cases of ``near independence'' and ``upper intermediate tail dependence''  exhibit asymptotic independence in the sense of \eqref{eq:bvasyind}; in this paper, we gain a better understanding of these ideas when $d>2$.

Note that for \emph{independence} of multiple random variables, it is well-known that ``pairwise independence'' for all pairs of random variables is not equivalent to their ``mutual independence'' (cf. \citet[Chapter 2]{hogg:mckean:craig:2013}). In the same vein, we propose here the concepts of  \emph{pairwise asymptotic independence} in \Cref{sec:pai} and \emph{mutual asymptotic independence} in \Cref{sec:mai}. With the new notion of mutual asymptotic independence, we explore the ideas of  ``near independence'' and ``intermediate upper tail dependence'' through all subsequent dimensions $2,3,\ldots,d$ going beyond just the $d$-dimensional characterization as given in \eqref{eq:uppertailorder}. 
{For models that lie between pairwise and mutually asymptotically independent models, we introduce the concept of \emph{$k$-wise asymptotic independence} for $k \in \{2, \ldots, d\}$ in \Cref{sec:kwiseasyind}.}
In particular,  we investigate and compare the various notions of \emph{asymptotic independence} and illustrate them using popular copula models.  Moreover, we obtain the following three key results for the popular Gaussian copula, which have broader theoretical and practical implications: 
\begin{itemize}
\item[(i)] a formulation of precise necessary and sufficient conditions for mutual asymptotic independence to hold,
\item[(ii)] a derivation of the correct tail orders, and,
\item[(iii)] the existence of Gaussian copula models exhibiting $k$-wise asymptotic independence but not $(k+1)$-wise asymptotic independence.
\end{itemize}
 Besides the Gaussian copula, we also provide examples to exhibit the breadth of asymptotic (in)dependence behavior using the Archimedean copula family.    
{We apply the new notions of asymptotic independence in \Cref{sec:dro} to show its implications on assessing the risk of complex systems under distributional ambiguity.
The different notions of asymptotic independence influence risk contagion in financial systems differently and hence may lead to an underestimation or overestimation of risk if applied improperly}. In particular, we exhibit this phenomenon using two pertinent conditional risk measures, namely, conditional tail probabilities and {Contagion Value-at-Risk or \revision{Conditional Value-at-Risk (CoVaR)}} in dimensions $d>2$. Finally, in \Cref{sec:concl}, we conclude with some broader implications of interpreting asymptotic independence in this new light. All proofs for the results presented in this paper are provided in the Appendix. 

\subsection*{Notations} 
We denote by $\mathbb{I}_d=\{1,\ldots,d\}$ an index set with $d$ elements and the cardinality of a set $S\subseteq \mathbb{I}_d$ is denoted by $|S|$. 
For a random vector $\bZ=(Z_1,\ldots,Z_d)$, we write $\bZ\sim F$ if $\bZ$ has distribution function $F$; moreover, we understand that marginally $Z_j\sim F_j$ for $j\in \mathbb{I}_d$. For any non-empty sets $S\subseteq\mathbb{I}_d$, the copula and survival copula of the corresponding $|S|$-dimensional marginal are denoted by $C_S$ and $\widehat{C}_S$, respectively. Moreover, if $d=1$ we have $C_S(u)=\widehat{C}_S(u)=u$ for $0\le u \le 1$. For a given vector $\bz\in\R^d$ and $S\subseteq \mathbb{I}_d$, we denote by  $\bz^{\top}$ the transpose of $\bz$ and by $\bz_S\in\R^{|S|}$ the vector obtained by deleting the components of $\bz$ in $\mathbb{I}_d\backslash S$. Similarly, for non-empty $S\subseteq\mathbb{I}_d$, $\Sigma_{S}$ denotes the appropriate sub-matrix of a given matrix $\Sigma\in\R^{d\times d}$ after removing all rows and columns with indices in $\mathbb{I}_d\setminus S$. Furthermore,  $\bzero_d=(0,\ldots,0)^\top$ and $\bone_d=(1,\ldots,1)^\top$ are vectors in $\R^d$, and $\II_d$ is the identity matrix in $\R^{d\times d}$; subscripts are dropped when evident from the context.
Vector operations are understood component-wise, e.g., for vectors $\bz=(z_1,\ldots,z_d)$ and $\by=(y_1,\ldots,y_d)$, $\bz\le \by$ means $z_j\le y_j$,  $\forall j\in \mathbb{I}_d$. 
For functions $f,g:(0,\infty)\to (0,\infty)$, we write $f(u)\sim g(u)$ as $u\downarrow0$ if $\lim_{u\downarrow 0}f(u)/g(u)=1.$ Moreover, a function $\ell:(0,\infty)\to(0,\infty)$ 
is called a \emph{slowly varying at $0$}, if $\lim_{u\downarrow0}\ell(uz)/\ell(u)=1, \forall z>0$.


\section{Pairwise asymptotic independence} \label{sec:pai}
Note that the definition in \eqref{eq:bvasyind} can be easily generalized to distributions with potentially unequal marginals; any random vector $(Z_1, Z_2)$ with continuous marginals $Z_j\sim F_j,$ $ j=1, 2$ is  asymptotically independent if
\begin{align}\label{eq:bvasyindcop}
    \widehat{C}(u,u) & = \P(F_1(Z_1)>1-u, F_2(Z_2)>1-u) \ = o(u), \quad u \downarrow 0,
\end{align}
where $\widehat{C}$ is the survival copula of $F$. 
Note that limit properties in \eqref{eq:bvasyind} and \eqref{eq:bvasyindcop} remain equivalent when the marginals of $(Z_1, Z_2)$ are \emph{completely tail equivalent}, i.e., $\P(Z_1>t)/\P(Z_2>t) \to 1$ as $t\to\infty$.  
Although not all extreme sets are of this form, this definition has been a key concept in the modeling of joint extremes.

An interesting feature of this definition of asymptotic independence is that it is based on tail sets tethered along the main diagonal $(t,t)$ (in \eqref{eq:bvasyind}) or $(1-u,1-u)$ (in \eqref{eq:bvasyindcop}). It is easy to check that \eqref{eq:bvasyindcop} is equivalent to
\[\widehat{C}(au,bu) = o(u), \quad u\downarrow 0,\]
for some  $a,b >0$ (\citet[Theorem 2]{balkema:nolde:2010a}). Curiously, an equivalent result for the distribution function of a bivariate random vector does not hold: even if \eqref{eq:bvasyind} holds it does not necessarily hold for diagonals of the form $(at,bt)$ for any $a, b >0$; see \citet[Proposition 3.9]{das:fasen:2023a} for an example with normally distributed marginals $(Z_1, Z_2)$ where $\P(Z_1>at, Z_2>t)=O(\P(Z_2>t)),$ as $t\to\infty$. 

 Although \eqref{eq:bvasyind} and \eqref{eq:bvasyindcop} are widely applied for bivariate random vectors, a proper multivariate characterization of \emph{asymptotic independence} has been relatively scarce. 
 A definition often used and based on all pairwise comparisons following \eqref{eq:bvasyindcop} is given next.
 \begin{definition}[Pairwise asymptotic independence]\label{def:pairai}
     A random vector $\bZ \in \R^d$ with \revision{ continuous marginal distributions and} survival copula $\widehat C$ is \textit{pairwise asymptotically independent }if $\forall\, j, \ell \in \mathbb{I}_d, j\neq \ell,$
   \begin{align}\label{eq:pairasyindcop}
    \widehat{C}_{\{j,\ell\}}(u,u) = o(u), \quad\quad u \downarrow 0. 
\end{align}
If $\bZ\sim F$ has copula $C$, we interchangeably say $\bZ, F, C$ or $\widehat{C}$ exhibits pairwise asymptotic independence. 
 \end{definition}
\revision{
\begin{remark} $\mbox{}$
\begin{enumerate}[(a)]
\item 
Note that for discrete distributions, the copula is no longer necessarily unique, so we cannot extend this definition straightforwardly. Therefore, we assumed continuous margins to receive the uniqueness of the copula. However,  in this paper, we are concerned with tail probabilities. 
Hence, it suffices to have the marginal distributions to be continuous above a fixed threshold (i.e., they are eventually continuous). For convenience, since we are primarily concerned with the asymptotic behavior of the survival copula, we will assume for the rest of the paper that all distributions have continuous marginal distributions.

\item    Note that in contrast to asymptotic independence, asymptotic upper tail dependence occurs for a $d$-dimensional survival copula $\widehat{C}$ with $d\ge2$ if 
     \begin{align}\label{eq:asydep}
     \lim_{u\downarrow 0}\frac{\widehat{C}(u,\ldots,u)}{u}= \revision{\lambda\in(0,1)}.
     \end{align}
    Obviously,  \eqref{eq:asydep} implies that \eqref{eq:pairasyindcop} cannot hold. With respect to the upper tail order defined in \eqref{eq:uppertailorder},  \eqref{eq:asydep} is equivalent to saying $\kappa=1$ and $\ell(u) \nrightarrow 0$ as $u\downarrow0$.
\end{enumerate}
\end{remark}
}

\subsection{Examples}
Pairwise asymptotic independence exists in many multivariate distributions. We note a few examples here. 
 \begin{example}[Independence]\label{ex:indpair} 
    If all components of a random vector $\bZ\in \R^d$  are independent, then of course,
    \begin{eqnarray*}
         \P(Z_1 > z_1,\ldots, Z_d > z_d) & = &\prod_{j=1}^d\ov F_j(z_j)
                 = \widehat C^{\text{ind}}(\ov F_1(z_1),\ldots,\ov F(z_d)) 
    \end{eqnarray*}
    for $(z_1,\ldots,z_d)\in\R^d$,  where $C^{\text{ind}}:[0,1]^d\to[0,1]^d$ is the \textit{independence copula} given by
    \begin{eqnarray} \label{independence copula}
        C^{\text{ind}}(u_1,\ldots,u_d)=\prod_{j=1}^du_j
    \end{eqnarray}
    for $0\leq u_1,\ldots, u_d \leq 1$ with survival copula 
    $\widehat C^{\text{ind}}(u_1,\ldots,u_d)=C^{\text{ind}}(u_1,\ldots,u_d)$. For any distinct $j,\ell \in \mathbb{I}_d$ the $(j,l)$ marginal survival copula is as well
    \[\widehat{C}_{\{j,\ell\}}^{\text{ind}}{(u_1,u_2)= u_1u_2, \quad 0\le u_1, u_2 \le 1}.\]
   Thus, clearly, \eqref{eq:pairasyindcop} holds, and hence, the independence copula exhibits pairwise asymptotic independence. 
\end{example}

\begin{example}[Marshall-Olkin dependence]\label{ex:mo:pai}
    The Marshall-Olkin distribution is used in reliability theory to capture the failure of subsystems in a networked system. Here we consider a particular Marshall-Olkin dependence; cf. \cite{lin:li:2014,das:fasen:2023}. Assume that for every non-empty set $S\subseteq \mathbb{I}_d$  there exists a parameter $\lambda_S>0$ and  $\Lambda:=\{\lambda_S: \emptyset\neq S\subseteq\mathbb{I}_d\}$. Then the generalized Marshall-Olkin  (MO) survival copula with rate parameter $\Lambda$ is given by
\begin{align}\label{eq:scop:mo}
   \widehat{C}^{\MO(\Lambda)}(u_1,\ldots,u_d) = \prod_{i=1}^d \prod_{|S|=i} \bigwedge_{j\in S} u_j^{\eta_j^S}
\end{align}
for $0\leq u_1,\ldots, u_d \leq 1$, where
\begin{align}\label{eq:MO:eta}
    \eta_j^S = {\lambda_S}\Big/{\Big(\sum\limits_{J \supseteq \{j\} } \lambda_J}\Big), \quad j\in S \subseteq \mathbb{I}_d.
\end{align}
\revision{For any distinct $j,\ell\in \mathbb{I}_d$, we can compute that
\begin{align*}
     \widehat{C}^{\MO(\Lambda)}_{\{j,\ell\}}(u,u)  = u^{\eta^*_{j\ell}}
\end{align*}
   with 
   \begin{align} \label{eq:eta*}
    \eta^*_{j\ell} & = \sum_{S\subseteq \mathbb{I}_d\atop  j\in S, \ell\notin S} \eta_j^S + \sum_{S\subseteq \mathbb{I}_d\atop  j\notin S, \ell\in S} \eta_{\ell}^S + \sum_{S\subseteq \mathbb{I}_d\atop   j,\ell \in S} \max\{\eta_j^S, \eta_\ell^S\}\revision{>1}.
\end{align}
Clearly, since $\eta_{j\ell}^*>1, \forall j\neq \ell$, $\widehat{C}^{\MO(\Lambda)}$  possesses pairwise asymptotic independence for any choice of 
    $\Lambda$; see as well \cite[Proposition 3]{lin:li:2014}. An even larger class of Marshall-Olkin copulas have been introduced in  \citet{lin:li:2014} which are also pairwise asymptotically independent.
}

Although $\lambda_S$ is allowed to take any positive value for non-empty $S\subset \mathbb{I}_d$,  we discuss below two particular interesting choices of the parameters, cf. \citet[Example 2.14]{das:fasen:2023}. 
\begin{enumerate}[(a)]
    \item Equal parameter for all sets: Here, $\lambda_S=\lambda$ for all non-empty $S\subseteq\mathbb{I}_d$ where $\lambda>0$ and we denote the survival copula by $ \widehat{C}^{\MO^{=}}$. We can check from \eqref{eq:MO:eta} that the value of $\lambda$ is irrelevant here. Clearly in this case $\eta_j^S = 1/2^{d-1},$ for all $j\in S$ and non-empty  $S\subset \mathbb{I}_d$. Hence we can compute the value of \revision{$\eta^*_{j\ell}$ } defined in \eqref{eq:eta*} as
    \begin{align*}
        \revision{\eta^*_{j\ell}} & \revision{= \eta^*_{12}}=\sum_{S\subseteq \mathbb{I}_d\atop  1\in S} \frac{1}{2^{d-1}} + \sum_{S\subseteq \mathbb{I}_d\atop  1\notin S, 2\in S}\frac{1}{2^{d-1}} =   \frac{2^{d-1}+2^{d-2}}{2^{d-1}} = \frac 32.
    \end{align*}
    Therefore, for all $j, \ell \in S$ with $j\not=\ell$,
    \begin{align*}
     \widehat{C}^{\MO^{=}}_{\{j,\ell\}}(u,u) & = u^{3/2}, \quad 0\le u \le 1.
\end{align*}
    \item Parameters proportional to the cardinality of the sets: Here, $\lambda_S=|S|\lambda$ for all non-empty $S\subseteq\mathbb{I}_d$ where $\lambda>0$ and we denote the survival copula by $ \widehat{C}^{\MO^{\infty}}$. As well the value of $\lambda$ is irrelevant and  for all $j\in S$ and non-empty subset $S\subset \mathbb{I}_d$ we have
    $$\eta_j^S = \frac{|S|}{(d+1)2^{d-2}}.$$ We compute again the value of \revision{$\eta^*_{j\ell}$ } defined in \eqref{eq:eta*} as
     \begin{align*}
        \revision{\eta^*_{j\ell} } & =\eta^*_{12}= \sum_{S\subseteq \mathbb{I}_d\atop  1\in S} \frac{|S|}{(d+1)2^{d-2}} + \sum_{S\subseteq \mathbb{I}_d\atop  1\notin S, 2\in S} \frac{|S|}{(d+1)2^{d-2}}\\
               & =   \frac{(d+1)2^{d-2}+d2^{d-3}}{(d+1)2^{d-2}} = 1 + \frac d{2(d+1)}.
    \end{align*}
Therefore, for all $j, \ell \in S$ with $j\not=\ell$,
    \begin{align*}
     \widehat{C}^{\MO^{\infty}}_{\{j,\ell\}}(u,u) & = u^{1+d/(2(d+1))}, \quad 0\le u \le 1.
\end{align*}
    \end{enumerate}
    The generalized MO copulas with these particular choices of parameters as in (a) and  (b) are also known as Caudras-Auge copulas \citep{cuadras:auge:1981} and have been used in L\'evy frailty models for survival analysis. Moreover, if the marginals are identically distributed, then the associated random vector turns out to be exchangeable \citep{durrett:1991}.
\end{example}

%

\begin{example}[Archimedean  copula] \label{ex:archcoppai}
A useful family of copula models for multivariate distributions is the Archimedean copulas \citep{joe:2014,Charpentier:Segers}. A $d$-dimensional copula $C^\phi$ is Archimedean if 
\begin{align}\label{eq:arch}
    C^\phi(u_1, \ldots, u_d) := \phi^{\leftinv}(\phi(u_1)+\ldots+\phi(u_d))
\end{align}
for $0\leq u_1,\ldots, u_d\leq1$, where the generator function $\phi:[0,1] \to [0,\infty]$ is convex, decreasing, with $\phi(1)=0$ and $\phi^{\leftinv}(y)=\inf\{u\in[0,1]: \phi(u)\le y\}$ for $y\in(0,\infty)$. Necessary and sufficient conditions on the function $\phi$ such that $C^\phi$ in \eqref{eq:arch} is a copula are given in \cite{mcneil:neslehova:2009}; note that the survival copula $\widehat C^\phi$ of an Archimedean copula $C^\phi$ is, in general, not Archimedean. A popular choice of $\phi$ is the Laplace transform of any positive random variable.

Tail dependence in such copulas has been studied in \cite{Charpentier:Segers, hua:joe:2011}, and sufficient conditions to obtain pairwise \revision{asymptotic} independence exist.
Suppose the random vector $\bZ$ having an Archimedean copula $C^{\phi}$ has a generator  $\phi$ satisfying
\begin{eqnarray*}
    \lim_{u\downarrow 0}\frac{u\phi'(1-u)}{\phi(1-u)}=1,
\end{eqnarray*}
then we may conclude from  \citet[Theorem 4.1 and equation (4.4)]{Charpentier:Segers}  that 
$\bZ$ is pairwise asymptotically independent. In contrast, if the limit
$$\theta_1:=\lim_{u\downarrow 0}\frac{u\phi'(1-u)}{\phi(1-u)}\in(1,\infty)$$ exists and is larger than $1$ then we have asymptotic upper tail dependence, i.e, $\widehat{C}^\phi$ satisfies \eqref{eq:asydep}.
We observe from \citet[Table 1]{Charpentier:Segers} that for many Archimedean copulas, we have $\theta_1=1$ and thus they are pairwise asymptotically independent; this includes Frank copula, Clayton copula, Ali-Mikhail-Haq copula, and so on; see also \citet[Table 4.1]{nelsen:2006} for further details. \end{example}

\begin{example}[Gaussian copula]\label{ex:gauss:pai}
The Gaussian dependence structure is perhaps the most popular one used in practice. Let $\Phi_\Sigma$ denote the distribution function of a $d$-variate  normal distribution with all marginal means zero, variances one and a positive-definite correlation matrix $\Sigma\in\R^{d\times d}$, and $\Phi$ denote a standard normal distribution function. Then for $0< u_1,\ldots, u_d < 1$,
$$C^{\Sigma}(u_1,\ldots,u_d)= \Phi_\Sigma(\Phi^{-1}(u_1),\ldots, \Phi^{-1}(u_d))$$
 denotes the Gaussian copula with correlation matrix $\Sigma$. 
Pairwise asymptotic independence has been well-known for the bivariate normal distribution, as well as the bivariate Gaussian copula  \revision{ if the correlation is less than one \citep{sibuya:1960,ledford:tawn:1996}.} 
\revision{ Hence, we may immediately conclude that for $d\geq 2$, a Gaussian copula $C^{\Sigma}$ exhibits pairwise asymptotic independence if $\Sigma$ is positive definite.}
In fact, it is possible to find the exact tail order for the Gaussian survival copula for any $S\subseteq \mathbb{I_d}$ with $|S|\ge 2$, the precise result is given in \Cref{subsec:gaussmai}. 

\end{example}

\section{Mutual asymptotic independence}\label{sec:mai}
 
Pairwise asymptotic independence has often either been used as a natural extension of asymptotic independence \citep{balkema:nolde:2010a, guillou:padoan:rizzelli:2018}, or taken as a consequence from other relevant properties \citet[Remark 6.2.5]{dehaan:ferreira:2006}, or implicitly assumed \citep{lalancette:engelke:volgushev:2021} in a variety of works. 
Next, we define a notion that captures the global joint concurrent tail behavior of random vectors portrayed by many popular multivariate dependence structures, e.g., dependence defined using Gaussian, Marshall-Olkin, or various Archimedean copulas, etc., but not restricted to mere replication of pairwise comparisons of tails. 
 \begin{definition}[Mutual asymptotic independence]\label{def:mai}
     A random vector $\bZ\in \R^d$  with \revision{ continuous marginal distributions and} survival copula $\widehat C$ is \emph{mutually asymptotically independent} if for all $S\subseteq\mathbb{I}_d$ with $|S|\ge 2$, we have
     \begin{align}\label{eq:mutualasyind}
     \lim_{u\downarrow 0} \frac{\widehat{C}_S(u,\ldots,u)}{\widehat{C}_{S\setminus \{\ell\}}(u,\ldots,u)} 
     = 0, \quad \forall\,\, \ell \in S,
     \end{align}
 where we define $0/0:=0$. If $\bZ\sim F$ has copula $C$, we interchangeably say $\bZ, F, C$ or $\widehat{C}$ possesses mutual asymptotic independence.
 \end{definition}

\revision{
\begin{remark}\label{rem:multasinfresnick}
 Some explanation is due here in order to distinguish between the traditional notion of \emph{multivariate asymptotic independence} (also called ``mutual" asymptotic independence sometimes) in dimensions $d>2$ (\citet[Chapter 5.5]{resnickbook:2008}, \citet[Chapter 5.2]{galambos:1978}, \citet[Chapter 7.6]{mcneil:frey:embrechts:2015}) and the notion defined in \Cref{def:mai}. Due to \citet[Proposition 5.27]{resnickbook:2008} under the constrain that the margins are continuous, identically distributed and in the maximum domain of attraction of a univariate extreme value distribution, pairwise asymptotic independence is equivalent to a distribution having ``multivariate asymptotic independence" meaning that the distribution lies in the maximum domain of attraction of an extreme value distribution with independent marginals (the limit distribution is a product measure). But our notions of ``mutual asymptotic independence'' and ``multivariate asymptotic independence'' are not equivalent which we see in the following and in particular,  in \Cref{ex:painomai}.
 \end{remark}
 }

When $d=2$, both \eqref{eq:pairasyindcop} and \eqref{eq:mutualasyind} boil down to \eqref{eq:bvasyindcop} and hence, are equivalent.    Assuming $d\ge 3$ and mutual asymptotic independence, if we take all choices of $S\subseteq\mathbb{I}_d$ with $|S|=2$, then \eqref{eq:mutualasyind} is just a restatement of  \eqref{eq:pairasyindcop}, implying pairwise asymptotic independence. \revision{We summarize this in the next proposition.}

\begin{proposition}\label{prop:maiimppai}
   If a random vector $\bZ\in \R^d,$ $ d\ge 2$ 
   \revision{with continuous marginal distributions}
   is mutually asymptotically independent, then it is also pairwise asymptotically independent.
\end{proposition}

The reverse implication of  \Cref{prop:maiimppai}  is not necessarily true as we see in the following example, which mimics the consequences for the analogous notions of classical ``mutual'' and ``pairwise independence'' (\cite{hogg:mckean:craig:2013}).

\begin{example}\label{ex:painomai}
The difference between pairwise and mutual independence can be shown using an $\R^3$-valued random vector with Bernoulli marginals  (cf. \citet[Chapter 2]{hogg:mckean:craig:2013}).
We take a similar spirit using uniform marginals. Consider i.i.d. uniform [0,1] random variables  $U, V, W$. Then $\revision{\bZ^*}=(U,V,W)$  is mutually asymptotically independent (cf. \Cref{ex:ind:mai}) and hence, pairwise asymptotically independent as well. Now consider $\bZ=(Z_1, Z_2, Z_3) \sim F$ such that
\begin{align*}
    \bZ & = \begin{cases}
         (U,V, \min(U,V)), & \text{ with prob. }\;\; 1/3,\\
          (U,\min(U,V),V), & \text{ with prob. } \;\; 1/3,\\
           (\min(U,V),U,V), & \text{ with prob. } \;\; 1/3.
    \end{cases}
\end{align*}
First note that for $0<z<1$, \revision{marginally}, 
\[F_j(z)= \P(Z_j\le z)=2z/3+1/3[1-(1-z)^2]= {4z/3 - z^2/3}, \quad j=1,2,3,\]
and hence, the $Z_j$'s are identically distributed.
 \revision{
\begin{itemize}
\item[(i)] If $\widehat{C}$ denotes the survival copula of $\bZ$, then we can check that for any $\{j,\ell\}\subset \{1,2,3\}$,
  \begin{align}
   \widehat{C}_{\{j,\ell\}}(u,u) &= \P(Z_j>2-\sqrt{1+3u}, Z_{\ell}>2-\sqrt{1+3u}) \nonumber\\
  & = \P(U>2-\sqrt{1+3u}, V>2-\sqrt{1+3u}) \nonumber\\ & =(\sqrt{1+3u}-1)^2=  9u^2/4 +o(u^2), \quad u\downarrow 0. \label{eq:expainomai:Cjl}
\end{align}
Hence, $\bZ$ exhibits pairwise asymptotic independence. 
\item[(ii)] But 
           {    \begin{align*}
  \widehat{C}(u,u,u) &= \P(U>2-\sqrt{1+3u}, V>2-\sqrt{1+3u})\\ & =(\sqrt{1+3u}-1)^2=  9u^2/4+o(u^2), \quad u\downarrow 0,
\end{align*}}
implying that $\bZ$ does not have mutual asymptotic independence.

\item[(iii)] We can compute that if $\bZ^{(1)}, \ldots, \bZ^{(n)}$ are i.i.d. $F$ and $\bM_n$ is the random vector of component-wise maxima given by
 $$\bM_n= 
 \left(\bigvee_{i=1}^n Z_1^{(i)},\bigvee_{i=1}^n Z_2^{(i)},\bigvee_{i=1}^n Z_3^{(i)}\right),$$
 then with $a_n=3/{2n}$ and $b_n=1$, we have for $\bx\in \R$,
 \begin{align}\label{eq:prodmeasGEV}
    \lim_{n\to\infty} \P(\bM_n \le a_n\bx+b_n) =   \lim_{n\to\infty} F^n(a_n\bx+b_n) = G(\bx) = \prod_{i=1}^3 \Psi_1(x_i).
 \end{align}
 where $\Psi_1(\cdot)$ is a univariate extreme value distribution given by $\Psi_1(x)=\min(e^x,1), x\in\R$. Thus $F\in \MDA(G)$ where $G$ is indeed a product measure according to \eqref{eq:prodmeasGEV}, implying multivariate asymptotic independence although $F$ does not have mutual asymptotic independence as shown in (ii).
 
 For illustration, we showed the multivariate asymptotic independence by hand, but the pairwise asymptotic independence in (i) and \citet[Proposition 5.27]{resnickbook:2008} already imply multivariate asymptotic independence. 
 \end{itemize}
}
\end{example}

\subsection{Examples: Part I}

It is instructive to note examples of mutual asymptotic independence in various distributions. 

 \begin{example}[Independence]\label{ex:ind:mai}
    Suppose $C^{\text{ind}}$ is the independence copula as given in \eqref{independence copula}, then the survival copula for any non-empty subset $S\subset\mathbb{I}_d$ satisfies
    \[\widehat{C}_S^{\text{ind}}(u,\ldots,u)= u^{|S|}, \quad\quad  0\le u \le 1.\]
   Thus, \eqref{eq:mutualasyind} holds for all such $S$ with $|S|\ge 2$ and hence, $C^{\text{ind}}$ exhibits mutual asymptotic independence. 
\end{example}

\begin{example}[Marshall-Olkin dependence]\label{ex:mo:mai}
In  \Cref{ex:mo:pai} we stated that any random vector $\bZ$ with dependence given by the generalized Marshall-Olkin survival copula $\widehat{C}^{\MO(\Lambda)}$ as defined in \eqref{eq:scop:mo} is pairwise asymptotically independent. 
\revision{In fact, \citet[Proposition 3]{lin:li:2014} allows us to conclude that $\widehat{C}^{\MO(\Lambda)}$ is indeed mutually asymptotically independent as well.}

\end{example}

\subsection{Examples: Part II}
In this section, we discuss examples that are pairwise asymptotically independent but sometimes are not mutually asymptotically independent. This will include a large class of examples from the Archimedean copula and the Gaussian copula family.

\subsubsection{Archimedean copula}
Recall the Archimedean copula $C^\phi$ defined in \Cref{ex:archcoppai}. The following result provides sufficient conditions on the generator $\phi$ for the random vector with Archimedean copula $C^\phi$ to possess both pairwise and mutual asymptotic independence.

\begin{theorem}[Archimedean  copula with mutual asymptotic independence] \label{Arch:mutuall}
Let the dependence structure of a random vector $\bZ\in\R^d$ \revision{with continuous marginal distributions } be given by an Archimedean copula $C^\phi$ with generator $\phi$ as in \eqref{eq:arch}. Suppose $\phi^{\leftarrow}$ is $d$-times continuously differentiable and $(-D)^{j}\phi^{\leftarrow}(0)<\infty$  $\forall\,j\in\mathbb{I}_d$.  Then $\bZ$
possesses both pairwise and mutual asymptotic independence.
\end{theorem}
The proof follows directly from \citet[Theorem 4.3]{Charpentier:Segers}.
The Archimedean copulas of \Cref{Arch:mutuall}  have the property that for any subset $S\subset \mathbb{I}_d$ with $|S|\ge 2$, the survival copula $\widehat C_S^{\phi}$ of the $|S|$-dimensional marginal behaves like the independence copula near the tails, i.e.,
\[\widehat C_S^{\phi}(u,\ldots,u) \sim u^{\kappa_S}, \quad u\downarrow 0, \]
where the upper tail order of $C_S$ is $\kappa_S=|S|$ (also follows from \cite[Theorem 4.3]{Charpentier:Segers}).
In particular, the upper tail order for $C^\phi$ is $\kappa=\kappa_{\mathbb{I}_d}=d$ and hence, these copulas are also ``nearly independent'' (see paragraph below \eqref{eq:uppertailorder}); several popular Archimedean copulas models, e.g. as Frank copula, Clayton copula and  Ali-Mikhail-Haq copula (\citet[Table 1]{Charpentier:Segers}) fall in this class exhibiting both pairwise and mutual asymptotic independence. In contrast, there are also Archimedean copulas exhibiting only pairwise asymptotic independence but not mutual asymptotic independence. The following result provides sufficient conditions on the generator $\phi$ to obtain such Archimedean copulas.

\begin{theorem}[Archimedean  copula with only pairwise asymptotic independence] \label{Arch:not:mutuall}
Let the dependence structure of a random vector $\bZ\in\R^d$ \revision{with continuous marginal distributions }  be given by an Archimedean copula $C^\phi$ with generator $\phi$ as in \eqref{eq:arch}. Suppose $\phi'(1)=0$ and
\begin{eqnarray*}
    L(u):=-\phi'(1-u)-u^{-1}\phi(1-u)
\end{eqnarray*}
is a positive function, which is slowly varying at $0$. 
Then $\bZ$ possesses pairwise asymptotic independence but does not possess mutual asymptotic independence.
\end{theorem}
The proof follows directly from \citet[Theorem 4.6 and Corollary 4.7]{Charpentier:Segers}.
Now, the Archimedean copulas of \Cref{Arch:not:mutuall} have a different characteristic in the sense that for any subset $S\subset \mathbb{I}_d$ with $|S|\ge 2$, the survival copula $\widehat C_S^{\phi}$ on the $|S|$-dimensional marginal behaves as 
\[\widehat C_S^{\phi}(u,\ldots,u) \sim u\ell(u), \quad u\downarrow 0, \]
where $\ell$ is a slowly varying function at 0 (follows from \citet[Corollary 4.7]{Charpentier:Segers}). Hence, the upper tail order of $C_S$ is $\kappa_S=1$ for all $S\subset \mathbb{I}_d$ with $|S|\ge 2$.
To obtain an example of such a copula,  take some  parameter  $\theta\in (0,\infty)$ and define the generator 
    $$\phi_\theta(u)=\frac{1-u}{(-\log(1-u))^\theta}, \quad 0\leq u\leq 1, $$
 of an Archimedean copula $C^{\phi_\theta}$. Then  $C^{\phi_\theta}$ satisfies the assumptions of \Cref{Arch:not:mutuall}, resulting in an Archimedean copula with pairwise but not mutual asymptotic independence; we refer to \citet[Table 1]{Charpentier:Segers}.

\subsubsection{Gaussian copula.}\label{subsec:gaussmai}
 In \Cref{ex:gauss:pai} we observed that any random vector with Gaussian copula having a positive definite correlation matrix has pairwise asymptotic independence. Interestingly, not all such models will have mutual asymptotic independence. The following theorem provides the exact condition for this. 

\begin{theorem}\label{thm:maiforGausscop}
  Let the dependence structure of a random vector $\bZ\in \R^d$ \revision{with continuous marginal distributions }  be given by a Gaussian copula $C^{\Sigma}$ with positive-definite correlation matrix $\Sigma$. 
   Then   $\bZ$ exhibits mutual asymptotic independence if and only if
      $\Sigma^{-1}_{S}\bone_{|S|} >\bzero_{|S|}$ for all non-empty sets $S\subseteq \mathbb{I}_d$. 
\end{theorem}

The proof of the theorem is quite involved, requiring a few auxiliary results based on the recently derived knowledge  on the asymptotic behavior of tail probabilities of a multivariate distribution with identically Pareto marginals and Gaussian copula $C^{\Sigma}$ in \citet{das:fasen:2023a}. \revision{Hence the proof has been relegated to \Cref{app:proofmaiGC}.} Curiously, the ingredients of the proof allow us to find the tail asymptotics of the survival copula of any $|S|$-dimensional marginal in terms of its tail order. 
\begin{proposition}\label{prop:Gausscoptailorder}
      Let 
      $C^{\Sigma}$ be a  Gaussian copula  with positive-definite correlation matrix $\Sigma$. 
   Then  
   for any subset $S\subset \mathbb{I}_d$ with $|S|\ge 2$, we have as $u\downarrow 0$,
   \begin{align}\label{eq:gcuptailorder}
       \widehat{C}_S^{\Sigma}(u,\ldots,u) \sim  u^{\kappa_S}
   \ell_{S}(u)
   \end{align}
   where $\ell_{S}$ is slowly varying at 0  and $$\kappa_S= \min\limits_{\{\bz\in\R^{|S|}:\bz\ge \bone_S\}} \bz^{\top}\Sigma_S^{-1}\bz.$$
\end{proposition}
\revision{A proof of this result is given in \Cref{app:proofmaiGC} as well.}


\begin{remark}
    A few interesting features of \Cref{prop:Gausscoptailorder} and related results  are to be noted here.
    \begin{enumerate}[(a)]
        \item Although \Cref{prop:Gausscoptailorder} only gives the tail order of  $\widehat{C}^{\Sigma}_S$, in fact, the exact tail asymptotics for $\widehat{C}^{\Sigma}_S(u\bv_S)$ as $u\downarrow 0$ for $\bv_S=(v_s)_{s\in S}$, $v_s\in(0,1)$, including the slowly varying function is available in \Cref{prop:gcsurvcop} in  \Cref{app:proofmaiGC}.
        \item The upper tail order $\kappa_S$ in \eqref{eq:gcuptailorder} is obtained as a solution to a quadratic programming problem; the exact solution is given in \Cref{lem:quadprog} in  \Cref{app:proofmaiGC}.
        \item With respect to \eqref{eq:gcuptailorder}, for subsets $S,T \subset \mathbb{I}_d$ with $S\subsetneq T$ and $|S|\ge 2$, it is possible to have (i) $\kappa_S<\kappa_T$, (ii) $\kappa_S=\kappa_T$, with $\ell_S(u)\sim c\,\ell_T(u), u\downarrow 0$ for $c>0$, and (iii) $\kappa_S=\kappa_T$, with $\ell_S(u)= o\left(\ell_T(u)\right), u\downarrow 0$. In \Cref{ex:gc3d}, we can observe both (i) and (ii) holding under different assumptions; an example for (iii) with Pareto marginals and Gaussian copula is available in \citet[Remark 5]{das:fasen:2023a}.
        \item \revision{In \citet[Example 1]{hua:joe:2011}, the authors already state that the tail order $\kappa$ of a Gaussian copula with positive definite correlation matrix $\Sigma$ is $\kappa=\bone^\top\Sigma^{-1}\bone$ (cf. \citet[Section 4.3.2]{joe:2014}). 
        However, to the best of our knowledge, the aforementioned paper does not specify that $\Sigma^{-1}\bone>\bzero$ is indeed a necessary condition for the result, since otherwise the statement is not valid; in fact, if $\Sigma^{-1}\bone\ngtr\bzero$ then  
        $\kappa<\bone^\top\Sigma^{-1}\bone$ is possible (cf. \Cref{lem:quadprog}).
        }
    \end{enumerate}
\end{remark}

\begin{example}\label{ex:gc3d}
For the purpose of illustration, we provide a positive-definite correlation matrix (with $d=3$) for a Gaussian copula parametrized by a single parameter $\rho$ which exhibits mutual asymptotic independence for only certain values of $\rho$ and only pairwise asymptotic independence for other feasible values;  see \citet[Example 1(b)]{das:fasen:2023a} for further details. Throughout, we denote by $\ell_j(u), j\in \mathbb{N}$,  a slowly varying function at 0. Consider the Gaussian copula $C^{\Sigma}$ with correlation matrix
\begin{align*}
    \Sigma=\begin{pmatrix}
1 & \rho & \sqrt{2}\rho\\
\rho & 1 & \sqrt{2}\rho \\
\sqrt{2}\rho  & \sqrt{2}\rho  & 1
\end{pmatrix},
\end{align*}
 where $\rho\in \left((1-\sqrt{17})/8, (1+\sqrt{17})/8\right)\approx (-0.39, 0.64)$ which ensures the positive definiteness of $\Sigma$. Clearly, pairwise asymptotic independence holds for all such $\rho$ values.
\begin{itemize}
 \item[(i)] Suppose $\rho<1/({2\sqrt{2}-1}) \approx 0.55$. Then one can check that $\Sigma^{-1}\bone>\bzero$, and hence, mutual asymptotic independence holds as well. In fact, we can find the behavior of the survival copula (using \Cref{prop:gcsurvcop} or, \citet[Example 1(b)]{das:fasen:2023a}): As $u\downarrow 0$,
 \begin{align}
     \widehat{C}^{\Sigma}(u,u,u) & \sim  u^{\frac{3-(4\sqrt{2}-1)\rho}{1+\rho-4\rho^2}} \ell_1(u). \label{eq:gc3rholt0p55}
 \intertext{We also find that as $u\downarrow 0$,}
 \begin{split}
     \widehat{C}_{\{13\}}^{\Sigma}(u,u) & =\widehat{C}_{\{23\}}^{\Sigma}(u,u) \sim  u^{\frac{2}{1+\sqrt{2}\rho}} \ell_2(u),  \quad \text{and,}\\ 
     \widehat{C}_{\{12\}}^{\Sigma}(u,u) & \sim  u^{\frac{2}{1+\rho}} \ell_3(u),      \end{split}\label{eq:gc3rholt0p55_M12}
 \end{align}
 \item[(ii)]
On the other hand, if $\rho \ge 1/({2\sqrt{2}-1})$, then $\Sigma^{-1}\bone \ngtr \bzero$ and the copula does not have mutual asymptotic independence.  Note that in this case, the behavior of the two-dimensional marginal survival copulas will still be given by \eqref{eq:gc3rholt0p55_M12}, but the tail behavior as seen in \eqref{eq:gc3rholt0p55} does not hold anymore. Now, as $u\downarrow 0$, we have
\begin{align*}
     \widehat{C}^{\Sigma}(u,u,u) & \sim  u^{\frac{2}{1+\rho}} \ell_4(u). 
 \end{align*}
 In fact we can check that $\ell_4(u)\sim \beta\,\ell_3(u)$ as  $u\downarrow 0$ for some constant $\beta>0$ (\citet[Example 1(b)]{das:fasen:2023a}), and hence
 \[ \widehat{C}^{\Sigma}(u,u,u) \sim \beta\, \widehat{C}_{\{12\}}^{\Sigma}(u,u), \quad u\downarrow 0,  \]
 also verifying that mutual asymptotic independence does indeed not hold here.
 \end{itemize}
\end{example}

\section{$k$-wise asymptotic independence}\label{sec:kwiseasyind}
The fact that some multivariate models exhibit pairwise asymptotic independence yet not mutual asymptotic independence naturally 
prompts the inquiry into the existence of models that lie in between these two notions. The following definition provides an answer.
\begin{definition}[$k$-wise asymptotic independence]\label{def:kai}
     A random vector $\bZ\in \R^d$ \revision{with continuous marginal distributions } and survival copula $\widehat C$ is \emph{$k$-wise asymptotically independent} for a fixed $k\in \{2,\ldots, d\}$, if for all $S\subseteq\mathbb{I}_d$ with $2\le |S|\le k$, we have
     \begin{align*}
     \lim_{u\downarrow 0} \frac{\widehat{C}_S(u,\ldots,u)}{\widehat{C}_{S\setminus \{\ell\}}(u,\ldots,u)} 
     = 0, \quad \forall\,\, \ell \in S,
     \end{align*}
 where we define $0/0:=0$. If $\bZ\sim F$ has copula $C$, we interchangeably say $\bZ, F, C$ or $\widehat{C}$ possesses $k$-wise asymptotic independence.
 \end{definition}

\revision{
\begin{remark}
The concept of $k$-wise asymptotic independence is a measure of dependence in the extremes. If a random vector $\bZ$ exhibits $k$-wise asymptotic independence but not $(k+1)$-wise asymptotic 
independence, then there exists a combination of at exactly $k$ components in $\bZ$, so that when these are large another additional component is large as well. Fewer than $k$ large components cannot produce a large value in another component.
Consequently, lower values of $k$ reflect a stronger dependence in the extremes.
\end{remark}
}

Note that for any $d$-dimensional copula, $d$-wise asymptotic independence is the same as mutual asymptotic independence (and of course \emph{2-wise} is the same as \emph{pairwise}). Again, following \Cref{prop:maiimppai}, we may check that mutual asymptotic independence indeed implies $k$-wise asymptotic independence for all $k\in \{2,\ldots,d\}$.  The converse of the previous implication is, of course, not true;  the examples in the following section also show this.

Obviously, an equivalent characterization of  $k$-wise asymptotic independence is the following.

\begin{proposition} \label{Remark 4.3}
A random vector $\bZ$ in $\mathbb{R}^d$ is $k$-wise asymptotically independent if and only if for all $S \subseteq \mathbb{I}_d$ with $|S| = k$, the random vector $\bZ_S$ in $\mathbb{R}^k$ is mutually asymptotically independent.
\end{proposition}

\subsection{Examples}

Indeed, within the class of Archimedean copulas as well as the class of Gaussian copulas with dimensions $d>2$, we find examples of models which exhibit $k$-wise asymptotic independence, but not $(k+1)$-wise asymptotic independence given any $k\in \{2, \ldots, d-1\}$. Consequently, these models are also not mutually asymptotically independent. Let us begin with an investigation of a particular Archimedean copula.

\subsubsection{ACIG copula}\label{ex:acig}
 This
 Archimedean copula based on the Laplace transform (LT) of an Inverse Gamma distribution, called ACIG copula in short, was introduced in \citet{hua:joe:2011}, i.e, if $Y=X^{-1}$ and $X\sim \text{Gamma}(\alpha,1)$ for $\alpha>0$, then the generator of this Archimedean copula is given by the LT of $Y$.
 The expression of the generator includes the Bessel function of the second kind. Closed-form expressions of the copula $C^\phi$ and survival copula $\widehat C^\phi$ are not easy to write down; nevertheless,  from computations in \citet[Example 4]{hua:joe:2011} and \citet[Example 4.4]{hua:joe:li:2014}, we can conclude that for any $d\ge 2$, the survival copula of the ACIG copula with parameter $\alpha>0$ has the following asymptotic behavior:
 \begin{align}\label{eq:acigtailorder}
     \widehat{C}^\phi(u,\ldots,u) \sim  \beta_d u^{\kappa_d}, \qquad u\downarrow 0,
 \end{align}
 where $\kappa_d=\max\{1,\min\{\alpha,j\}\}$ and $\beta_d>0$ is a positive constant. Here, $\kappa_d$ is the tail order of the copula. Therefore, if $\alpha\le 1$ then $\kappa_d=1$ for all $d\ge 2$, and if  $\alpha>1$, then $\kappa_d=\min(\alpha,d)$. Note that by the exchangeability property of Archimedean copulas and \eqref{eq:acigtailorder}, we know that for any $S\subset\mathbb{I}_d$ with $|S|\ge 2$,
  \begin{align*}
     \widehat{C}_S^\phi(u,\ldots,u)  \sim  \beta_{|S|} u^{\kappa_{|S|}}, \quad\ u\downarrow 0.
 \end{align*} 
 Thus, we may conclude that for an ACIG copula with parameter $\alpha>0$, the following holds:
 \begin{enumerate}[(i)]
     \item If $0<\alpha\le 1$, the ACIG copula  exhibits asymptotic upper tail dependence.
     \item If $1<\alpha\le d-1$, the ACIG copula exhibits pairwise asymptotic independence but not mutual asymptotic independence. If additionally $k-1<\alpha\leq k$ for $k\in\{2,\ldots, d-1\}$, then the ACIG copula still exhibits $i$-wise asymptotic independence for all $i\in\{2,\ldots, k\}$, but not $(k+1)$-wise asymptotic independence.
     \item If $\alpha>d-1$, the ACIG copula exhibits $k$-wise asymptotic independence for all $k\in \{2, \ldots, d\}$, {and hence mutual asymptotic independence as well.}
 \end{enumerate}

\subsubsection{Gaussian copula}\label{ex:acig}
{The Gaussian copula has been popular in modeling dependence in a wide variety of applications. It turns out that a class of Gaussian copula models is also able to capture the presence of $k$-wise asymptotic independence and not $(k+1)$-wise asymptotic independence. This is demonstrated in the following result, \revision{whose proof is given in \Cref{Appendix B}.}}

\begin{theorem}  \label{Theorem 4.4}
  \revision{Suppose $k \in \{2, \ldots, d-1\}$ and $S_1\subseteq S_2\subseteq \{1,\ldots,d\}$ with $\vert S_1\vert=k$ and  $\vert S_2\vert=k+1$. Then there exists 
  a Gaussian copula $C^{\Sigma}$ } and a positive-definite correlation matrix $\Sigma$, such that $C^{\Sigma}$ exhibits $k$-wise asymptotic independence but not $(k+1)$-wise asymptotic independence and for any $x>0$,
  \begin{eqnarray} \label{eq.4.2}
        \lim_{u\downarrow 0}\frac{\widehat C^{\Sigma}_{S_2}(u,\ldots,u,xu,u,\ldots,u)}{\widehat C^{\Sigma}_{S_1}(u,\ldots,u)}=1, 
  \end{eqnarray}
  where $xu$ is placed at the unique element in $S_2\backslash S_1$.
\end{theorem}

This theorem not only provides the existence of a $k$-wise asymptotically independent Gaussian copula, but it also gives the striking feature of the copula behavior in \eqref{eq.4.2} where, surprisingly the value of $x$ has no influence. \revision{
This means that for a random vector $\bZ$ with Gaussian copula $C^{\Sigma}$, as given in \Cref{Theorem 4.4}, and identically distributed marginals,
large values in all the $S_1$ components result in an extremely large value in the single component of $S_2\backslash S_1$; hence there is a strong dependence between the extremes of the components of $S_1$ and that of  $S_2\backslash S_1$.
All components in $S_1$ must be large at the same time; only a few components that are large do not result in a large value in the $S_2\backslash S_1$ component. This is demonstrated quite nicely in the following example, which was used in the proof of \Cref{Theorem 4.4}. 

\begin{example}
Let the correlation matrix of the Gaussian copula be given as
\begin{eqnarray*}
	\Sigma=\begin{bmatrix}
		\II_{d-1}\;\;\; & \rho \bone_{d-1}\\[1em]
           \rho \bone_{d-1}^{\top} & 1
\end{bmatrix}
\end{eqnarray*}
for some $\rho\in(0,1)$. Then 
for $\rho\in(1/(k-1),1/(k-2))$ the Gaussian copula $\widehat C^{\Sigma}$ is $(k-1)$-wise asymptotically independent but not  $k$-wise asymptotically independent (see proof of \Cref{Theorem 4.4}); therefore, if the first $(k-1)$ components are jointly large then as well the last component is large. But if we consider fewer components than $(k-1)$ components to be large, they have no effect on the size of the last component. 
Note that here $k-1=\lceil \rho^{-1}\rceil=\inf\{m\in \N:\rho^{-1}\leq m\}$. Thus, for a high value of $\rho$, fewer components, namely only the first $\lceil \rho^{-1}\rceil$ components, result in a large value in the last component. Although $\rho$ provides a measure of linear dependence, in this example, it also measures the dependence in the extremes which is reflected in the degree $\lceil \rho^{-1}\rceil$ of asymptotic independence.
\end{example}
}

\medskip

For the derivation of worst-case measures for risk contagion under distributional ambiguity in the next section, \revision{\Cref{Theorem 4.4}} turns out to be indispensable.

\begin{remark}
It is important to mention here that although {most popular copula families are bivariate in nature}, \citet[Chapter 4]{joe:2014} lists multiple extensions of bivariate copulas to general high dimensions; many such copulas can be explored for creating models with particular levels of asymptotic independence as necessitated by the context.  
\end{remark}

\section{Implication on risk management under distributional ambiguity}\label{sec:dro}
In financial risk management, a variety of risk measures are used to assess the risk contagion between different financial products, including stocks, bonds and equities. Such contagion or systemic risk measures are often based on conditional probabilities and range from computing regular conditional tail probabilities to CoVaR, marginal expected shortfall (MES), marginal mean excess (MME), and more; see \cite{das:fasen:2018, adrian:brunnermeier:2016} for details.  

Here, we focus on two such measures of risk contagion based on specific conditional tail probabilities and conditional tail quantiles. First, recall that for a random variable $Z$, the \emph{Value-at-Risk} or $\VaR$ at level $1-\gamma\in (0,1)$ is defined as 
\begin{align*}
\VaR_{\gamma}(Z) &:= \inf\big\{y\in\R: \P(Z>y)\le \gamma\big\}=\inf\big\{y\in\R: \P(Z\le y) \geq 1-\gamma\big\},
\end{align*} 
the $(1-\gamma)$-quantile of $Z$ where $\inf \emptyset:=\infty$  (cf. \cite{Embrechts:Hofert:Chavez}).
If $Z\sim F$ is continuous, then $F$ is invertible with inverse $F^{-1}$ and for $0<\gamma<1$ we have $\VaR_{\gamma}(Z)= F^{-1}(1-\gamma)$. 

Consider the returns from a portfolio of $d>1$ stocks being given by the random vector $\bZ=(Z_1,\ldots, Z_d)\sim F$. Suppose we are interested in measuring the risk of $Z_1$ having an extremely large value, given that all variables in some {non-empty} subset $J\subset \mathbb{I}_d\backslash \{1\}$ with $\vert J\vert =\ell$ are at extremely high levels. This can be captured via
the following \emph{conditional tail probability} 
    \begin{align}\label{eq:condprob}
            \P\left(Z_1>t | Z_j>t,  \,\forall j\,\in J\right),
    \end{align}
    as $t\to\infty$. 
    Alternatively, for a level $\gamma\in (0,1)$, we are interested in the risk measure
    \begin{align}\label{eq:CTP}
        \CTP_{\gamma}(\bZ_{1|J}) :=\P\left(Z_1>\VaR_{\gamma}(Z_1) | Z_j>\VaR_{\gamma}(Z_j), \,\forall j\,\in J\right)
    \end{align}
     as $\gamma\to0$.
  Note that \eqref{eq:condprob} and \eqref{eq:CTP} are equivalent if all the marginal random variables $Z_1,\ldots, Z_d$ are identically distributed. For convenience, we will focus on the measure $\CTP_{\gamma}$ as defined in \eqref{eq:CTP}.
  
A second measure of risk contagion we are interested in is a generalization of  the $\VaR$ to the multivariate setting given by the {\emph{Contagion Value-at-Risk or CoVaR}} at confidence level $(\gamma_1,\gamma_2)$ for $\gamma_1,\gamma_2\in(0,1)$ defined as
  \begin{eqnarray}\label{eq:covar}
    \text{CoVaR}_{\gamma_1,\gamma_2}(\bZ_{1|J}):=\inf\{z\in\R_+:
    \P(Z_1>z\vert Z_j>\VaR_{\gamma_2}(Z_j),\,\forall j\in J)\leq \gamma_1\}.
\end{eqnarray}
 The risk measure CoVaR was introduced in the bivariate setting for $J={2}$ to capture risk contagion, as well as systemic risk by \citet{adrian:brunnermeier:2016} where they used the conditioning event to be $Z_2=\VaR_{\gamma_2}(Z_2)$; this was later modified by \citet{girardi:ergun:2013} to $Z_2>\VaR_{\gamma_2}(Z_2)$ with the restriction that $\gamma_1=\gamma_2$; this latter definition has been widely used in dependence modeling \citep{nolde:zhou:zhou:2022,Mainik:Schaaning,hardle:wang:yu:2016,das:fasen:2024} {and is generalized in our definition given in \eqref{eq:covar}. }
  


In risk management applications,  computing quantities like CTP and CoVaR requires knowledge of the joint distribution of the risk vector $\bZ$.
Even if the univariate distributions of all the marginal variables can be estimated, the joint distribution often remains unknown and relatively more involved for estimation purposes.
 An approach often used is to provide a worst-case value for such risk measures under certain constraints on the joint distribution of the variables. Naturally, for such tail risk measures, constraints can be provided in terms of their joint asymptotic tail behavior, including pairwise, mutual or $k$-wise asymptotic independence. It turns out that under different constraints, we may obtain a different tail behavior for the worst-case measures. To further this discussion, let us define $\mathcal{P}$ to be the class of all probability distributions in $\R^d$ with continuous marginal distributions. 
\revision{%
For $k\in \{2,\ldots,d\}$ define the classes of distributions
\begin{align*} 
    \mathcal{P}_k &:= \{F\in \mathcal{P}: F \text{ possesses  $k$-wise asymptotic independence}\},
\intertext{and similarly, the restrictions to distributions with Gaussian copulas}
 \mathcal{N}_k&:=\mathcal{P}_k\cap\{F\in \mathcal{P}:F\text{ has a Gaussian copula }C^{\Sigma} \text{ with } \Sigma \text{ positive definite}\}.
\end{align*}
}
Note that $\mathcal{P}_2$ models the class of pairwise asymptotically independent random vectors, whereas $\mathcal{P}_d$ models the class of mutually asymptotically independent random vectors. By Definition \ref{def:kai}, it is easy to check that  
\begin{align}\label{eq:Pinclusions}
  \mathcal{P} \supseteq \mathcal{P}_2 \supseteq \mathcal{P}_3 \supseteq \cdots \supseteq{\mathcal{P}}_d \quad \text{ and } \quad \mathcal{P} \supseteq \mathcal{N}_2 \supseteq \mathcal{N}_3 \supseteq \cdots \supseteq{\mathcal{N}}_d.  
\end{align}
Furthermore, these classes are non-empty, since 
$\mathcal{N}_k\neq\emptyset $ by \Cref{Theorem 4.4}
and $\mathcal{N}_k\subseteq \mathcal{P}_k$.



Since the joint distributions are unknown we may want to find the worst case $\CTP$ or CoVaR in such cases where $F\in \mathcal{P}_k \subset \mathcal{P}$ or $F\in \mathcal{N}_k$, $k\in\{2,\ldots,d\}$. First, we present the result for the $\CTP$. \revision{The proof of this theorem and all subsequent results in this section are given in \Cref{Appendix C}.}

\begin{theorem} \label{Theorem 5.1}
    Let $\bZ=(Z_1,\ldots,Z_d)\sim F$, $d\ge2$, has continuous marginal distributions. 
    Furthermore, suppose $J\subset \mathbb{I}_d\backslash \{1\}$ with $\vert J\vert =\ell$. 
    \begin{enumerate}[(a)]
        \item If $k\in\{\ell+1,\ldots,d\}$, then
    \begin{eqnarray*}
     \sup_{F\in \mathcal{N}_k}\lim_{\gamma\downarrow 0} \CTP_{\gamma}(\bZ_{1|J})= \sup_{F\in \mathcal{P}_k}\lim_{\gamma\downarrow 0} \CTP_{\gamma}(\bZ_{1|J}) =0. 
     \end{eqnarray*}
     \item If $k\in \{2,\ldots,\ell\}$, then
       \begin{eqnarray*}          
        \sup_{F\in \mathcal{N}_k}\lim_{\gamma\downarrow 0} \CTP_{\gamma}(\bZ_{1|J}) =\sup_{F\in \mathcal{P}_k}\lim_{\gamma\downarrow 0} \CTP_{\gamma}(\bZ_{1|J}) =1. 
        \end{eqnarray*}
    \end{enumerate}
\end{theorem}

The results indicate a qualitatively different behavior of the worst-case CTP depending on whether the tail dependence exhibits $k$-wise asymptotic independence with $k > \vert J\vert$ vis-a-vis $k \leq \vert J\vert$. When $k > \vert J\vert$, $\CTP_{\gamma}(\bZ_{1|J})$ converges to $0$ as $\gamma \downarrow 0$, suggesting that extreme large losses of $Z_j$ for all $j \in J$ have a negligible influence on extreme large losses of $Z_1$. In contrast, when $k \leq \vert J\vert$, there exists a $k$-wise asymptotically independent distribution function $F$, which is also pairwise asymptotically independent, such that extremely large losses of $Z_j$ for all $j \in J$ result, with a probability converging to $1$, in an extremely large loss of $Z_1$. In particular, for the Gaussian copula that is an astonishing result because it is in contrast to the belief that there are no joint extremes. This shows 
that for measuring risk contagion it is important to distinguish between these different concepts of tail independence 
and assuming an improper notion of asymptotic independence for our risk portfolio may lead to either underestimation or overestimation of the risk contagion.

In the following, we investigate the asymptotic behavior of the measure $\CoVaR$. For technical reasons, we restrict the class 
$\mathcal{P}_k$ slightly, in particular we will assume that $F_1$, the distribution of $Z_1$, is Pareto distributed, i.e., $F_1(z)=1-z^{-\alpha}$, $z\geq 1$, for some $\alpha>0$. 
\revision{Suppose
$\mathcal{P}^*:=\{F\in \mathcal{P}:F_1\text{ is Pareto distributed}\}$.
For $k\in \{2,\ldots,d\}$ define the classes  
\begin{align*}
    \mathcal{P}^*_k&:=\mathcal{P}_k\cap\left\{F\in \mathcal{P}^*:\sup_{\gamma\in(0,x^{-1}]}\frac{\widehat C_S(x\gamma,\gamma,\ldots,\gamma)}{\widehat C_S(\gamma,\gamma,\ldots,\gamma)}<\infty, \,\forall\, S\subseteq \mathbb{I}_d, \forall\,x\geq 1\right\}\subseteq \mathcal{P}_k,
\intertext{and}
    \mathcal{N}_k^*&:=\mathcal{N}_k\cap  \mathcal{P}^*\subseteq \mathcal{N}_k.
\end{align*}
}

\begin{remark}
 Instead of assuming that $F_1$ follows a Pareto distribution, it is possible to consider a broader class, allowing $F_1$ to have a regularly varying tail. {However, this approach makes the proofs more technical without providing any further valuable insights, hence, we have exhibited our results for the smaller class $\mathcal{P}_k^*$ for the purpose of exposition.}
\end{remark}

Although we reduce the class  $\mathcal{P}_k$ to $\mathcal{P}_k^*$, it still remains quite large and contains, in particular, $k$-wise asymptotically independent Gaussian copulas (with $F_1$ Pareto distributed).
\begin{lemma} \label{Lemma 5.3}
    $\mathcal{N}_k^*\subseteq \mathcal{P}_k^*$ 
     for $k\in \{2,\ldots,d\}$.
\end{lemma}


By restricting our consideration to the sets $\mathcal{P}^*_k$ and $\mathcal{N}^*_k$, we derive the subsequent result concerning the asymptotic behavior of the $\CoVaR$.

\begin{theorem} \label{Theorem 5.4}
    Let $\bZ=(Z_1,\ldots,Z_d)\sim F$, $d\ge2$, has continuous marginal distributions  and $F_1$ is a Pareto distribution. Furthermore, let $J\subset \mathbb{I}_d\backslash \{1\}$ with $\vert J\vert =\ell$. 
    \begin{enumerate}[(a)]
        \item If $k\in\{\ell+1,\ldots,d\}$, then for any $\gamma_1\in(0,1)$,
    \begin{eqnarray*}
     \sup_{F\in \mathcal{N}_k^*}\lim_{\gamma_2\downarrow 0} 
     \frac{\text{CoVaR}_{\gamma_1,\gamma_2}(\bZ_{1|J})}{\VaR_{\gamma_2}(Z_1)}=\sup_{F\in \mathcal{P}^*_k}\lim_{\gamma_2\downarrow 0} 
     \frac{\text{CoVaR}_{\gamma_1,\gamma_2}(\bZ_{1|J})}{\VaR_{\gamma_2}(Z_1)}=0.
     \end{eqnarray*}
     \item If $k\in \{2,\ldots,\ell\}$, then for any $\gamma_1\in(0,1)$,
       \begin{eqnarray*}          
        \sup_{F\in \mathcal{N}_k^*}\lim_{\gamma_2\downarrow 0} 
     \frac{\text{CoVaR}_{\gamma_1,\gamma_2}(\bZ_{1|J})}{\VaR_{\gamma_2}(Z_1)}=\sup_{F\in \mathcal{P}_k^*}\lim_{\gamma_2\downarrow 0} 
     \frac{\text{CoVaR}_{\gamma_1,\gamma_2}(\bZ_{1|J})}{\VaR_{\gamma_2}(Z_1)}=\infty. 
        \end{eqnarray*}
    \end{enumerate}
\end{theorem}
Akin to the case of finding for the worst-case CTP, we observe that the worst-case CoVaR also has a qualitatively different behavior depending if the tail dependence exhibits $k$-wise asymptotic independence with $k> \vert J\vert$, or with $k\leq  \vert J\vert$. 
When 
$k>\vert J\vert$,   the
ratio  $\text{CoVaR}_{\gamma_1,\gamma_2}(\bZ_{1|J})/{\VaR_{\gamma_2}(Z_1)}$ converges to $0$, reflecting 
that $\text{CoVaR}_{\gamma_1,\gamma_2}(\bZ_{1|J})$ increases at a negligible rate in comparison to $\VaR_{\gamma_2}(Z_1)$ as $\gamma_2\downarrow 0$ and that $\text{CoVaR}_{\gamma_1,\gamma_2}(\bZ_{1|J})$ is relatively small, i.e, the required risk reserve capital is low. 
But if $k\leq\vert J\vert$, there exists a $F\in\mathcal{N}_k^*\subseteq\mathcal{P}_k^*$ where $\text{CoVaR}_{\gamma_1,\gamma_2}(\bZ_{1|J})/{\VaR_{\gamma_2}(Z_1)}$ converges to $\infty$, so that $\text{CoVaR}_{\gamma_1,\gamma_2}(\bZ_{1|J})$ may increase much faster to $\infty$ than $\VaR_{\gamma_2}(Z_1)$ as $\gamma_2\downarrow 0$, giving a relatively high $\text{CoVaR}_{\gamma_1,\gamma_2}(\bZ_{1|J})$ and a higher reserve risk capital requirement.

\begin{remark}
Computations analogous to the ones carried out in this section, can also be done for other measures of risk contagion, for example, the marginal expected shortfall (MES), or, the marginal mean excess (MME) \citep{cai:einmahl:dehaan:zhou:2015, das:fasen:2018}; but, similar to the case of computing CoVaR, we need to restrict $\mathcal{P}_k$ to smaller classes satisfying various technical conditions. We leave these pursuits for the interested researchers to explore in the future.
\end{remark}

\section{Conclusion}\label{sec:concl}
In this paper, we provide a notion of multivariate asymptotic independence that is useful in comparing extreme events in different dimensions beyond mere pairwise comparisons, \revision{which has been traditionally used in the literature}.  This parallels the dichotomy of mutual independence vis-a-vis pairwise independence for multivariate random vectors.
We believe this new notion also provides an alternate pathway for characterizing extremal dependence for high-dimensional problems relating to tail events. 
We have illustrated using examples of particular copula models, including a few from the Archimedean family along with the Gaussian and the Marshall-Olkin copula. The copulas considered often exhibit at least pairwise asymptotic independence if not mutual asymptotic independence. For both Archimedean and Gaussian copulas, we presented examples exhibiting not only mutual asymptotic independence but also exhibiting only pairwise asymptotic independence but not mutual asymptotic independence. In particular, for the Gaussian copula, this result is quite striking since it is in contrast to the common belief that the Gaussian copula does not allow joint extremes. We have also introduced the concept of $k$-wise asymptotic independence which generalizes these two notions \revision{(pairwise and mutual)} and brings them under the same umbrella. Here we have shown that for any $k\in \{2,\ldots,d\}$ there exists a $k$-wise asymptotically independent Gaussian copula (which is not $(k+1)$-wise asymptotically independent if $k<d$). \revision{Moreover, we have exhibited that these} assumptions of different notions of asymptotic tail independence significantly impact measures of risk contagion within a financial system, such as conditional tail probabilities (CTP) or {Contagion} Value-at-Risk (CoVaR), depending on the specific context. Overlooking these concepts and assuming merely pairwise asymptotic independence for models \revision{may often} lead to a \revision{significant} underestimation of risks.

\bibliographystyle{imsart-nameyear}
\bibliography{bibfilenew}

\appendix

\section{Proofs of Section \ref{sec:mai}  }\label{app:proofmaiGC}

First, we present some auxiliary results required for the proof of 
\Cref{thm:maiforGausscop}.
The following lemma is from \citet[Proposition 2.5 and Corollary 2.7]{hashorva:husler:2002}.  

\begin{lemma}\label{lem:quadprog}
Let $\Sigma\in \R^{d\times d}$ be a positive-definite correlation matrix. Then for any $S\subset \mathbb{I}_d$ with $|S|\ge 2$, the quadratic programming problem
\begin{align*} 
  \mathcal{P}_{\Sigma^{-1}_S}: \min_{\{\bz\in\R^{|S|}:\bz\ge \bone_S\}} \bz^\top \Sigma^{-1}_S\bz
\end{align*}
  has a unique solution $\be^S\in\R^d$ such that
    \begin{align*} 
     \kappa_{S}:=\min_{\{\bz\in\R^{|S|}:\bz\ge \bone_S\}} \bz^\top \Sigma_S^{-1} \bz=\be^{S\,\top} \Sigma_S^{-1}\be^{S}>1.
\end{align*}
Moreover, there exists a unique non-empty index set  $I_{S}\subseteq S$ with  $J_S:=S\setminus I_S$ such that  the unique solution $\be^S$ is given by
\begin{align*}
    \be^S_{I_S}&=\bone_{I_S}, \\
    \be^S_{J_S}&=-[\Sigma_S^{-1}]_{J_SJ_S}^{-1}[\Sigma_S^{-1}]_{J_SI_S}\bone_{I_S}
    \geq \bone_{J_S},
\end{align*}
and $ \bone_{I_S}\Sigma_{I_S}^{-1}\bone_{I_S}=\be^{S\,\top} \Sigma_S^{-1}\be^{S}=\kappa_S>1$ as well as  $\bz^{\top}\Sigma_S^{-1}\be^S=\bz_{I_S}^{\top}\Sigma_{I_S}^{-1}\bone_{I_S} $ $\forall\,\bz\in\R^{|S|}$. 
Also defining   $h_i^S:=e_i^\top \Sigma_{I_S}^{-1}\bone_{I_S}$ for $i\in I_S$ where $e_i$ has only one non-zero entry 1 at the $i$-th co-ordinate, we have $h_i^S>0 \; \forall i\in I_S$.
\end{lemma}


\begin{lemma} \label{lem:aux1}
Let $\Sigma\in\R^{d\times d}$ be a positive definite correlation matrix and $I:=I_{\mathbb{I}_d}$ be defined as in \Cref{lem:quadprog}.
\begin{enumerate}
    \item[(a)] Suppose $\Sigma^{-1}\bone>\bzero$. Then for any $S\subseteq \mathbb{I}_d$ with $S\not= \mathbb{I}_d$, the inequality
    $
        \kappa_{\mathbb{I}_d}>\kappa_{S}$ holds.
    \item[(b)] Suppose $\Sigma^{-1}\bone\ngtr\bzero$. Then $I\neq \mathbb{I}_d$ and for any set $  S\neq \mathbb{I}_d$ with $I\subseteq S\subseteq \mathbb{I}_d$ the 
    equality
        $\kappa_{\mathbb{I}_d}=\kappa_S$ holds.
    For $  S\subseteq  \mathbb{I}_d$ with $S^c\cap I\not=\emptyset $ we have  $I=I_S$ and the inequality
    $ \kappa_{\mathbb{I}_d}>\kappa_S$ holds.
\end{enumerate}
\end{lemma}

\begin{proof}
We start with some preliminary calculations. Suppose $ S\subseteq \mathbb{I}_d$ with $S^c\cap I\not=\emptyset $. Let $\be^*:=\be^{\mathbb{I}_d}$ be the unique solution of the quadratic programming problem $\mathcal{P}_{\Sigma^{-1}}$ such that  $\kappa_{\mathbb{I}_d}=\be^{*\top }\Sigma^{-1}\be^*$, $\be^*\geq\bone $
  and $[\Sigma^{-1}\be^*]_{S^c}\not=\bzero_{S^c}$  since $[\Sigma^{-1}\be^*]_{I}>\bzero_{I}$
  and $S^c\cap I\not=\emptyset $ (cf. \Cref{lem:quadprog}).
  First, define $\widetilde{\be}_S:=\be_{S^c}^{*}+[\Sigma^{-1}]_{S^c}^{-1}[\Sigma^{-1}]_{S^cS}\be_S^{*}$ and note that
  \begin{eqnarray} \label{eq1.1}
        \widetilde{\be}_S&=&\be_{S^c}^{*}+[\Sigma^{-1}]_{S^c}^{-1}[\Sigma^{-1}]_{S^cS}\be_S^{*} \nonumber \\
        &=&[\Sigma^{-1}]_{S^c}^{-1}\left([\Sigma^{-1}]_{S^c}\be_{S^c}^{*}
        +[\Sigma^{-1}]_{S^cS}e_S^{*}\right) \nonumber\\
        &=&[\Sigma^{-1}]_{S^c}^{-1}\left[\Sigma^{-1}\be^{*}\right]_{S^c}\not=\bzero_{S^c}.
  \end{eqnarray}
  Finally, the Schur decomposition (see \citet[eq. (B2)]{Lauritzen})
  \begin{eqnarray*}
        [\Sigma^{-1}]_{S}=\Sigma_{S}^{-1}+[\Sigma^{-1}]_{SS^c}[\Sigma^{-1}]_{S^c}^{-1}[\Sigma^{-1}]_{S^cS}
  \end{eqnarray*}
  along with \eqref{eq1.1} imply that
   \begin{eqnarray}
        \kappa_{\mathbb{I}_d}
        &=& \be^{*\top}\Sigma^{-1}\be^* \nonumber\\
             &=& \be_S^{*\top}\Sigma_S^{-1}\be^*_S  + \widetilde{\be}_S^{\top}[\Sigma^{-1}]_{S^c}\widetilde{\be}_S \label{Schur}\\
            &> & \be_S^{*\top}\Sigma^{-1}\be^*_S\geq \min_{\bz_S\geq \bone_S}
            \bz_S^{\top}\Sigma^{-1}_S\bz_S=\kappa_S. \label{eq2}
  \end{eqnarray}

\noindent (a)  If $\Sigma^{-1}\bone>\bzero$ then $I=\mathbb{I}_d$
    and $\be^*=\bone$; see \citet[Proposition 2.5]{hashorva:husler:2002}.
    Thus, any $S\subseteq\mathbb{I}_d$ with $S\neq \mathbb{I}_d$  satisfies   $S^c\cap I\not=\emptyset $ and the result follows from \eqref{eq2}.\\
\noindent (b) If $\Sigma^{-1}\bone\ngtr\bzero$ then $I\subseteq\mathbb{I}_d$  and $I\neq \mathbb{I}_d$; see \citet[Proposition 2.5]{hashorva:husler:2002}. Hence,
    \Cref{lem:quadprog} and $\Sigma_I^{-1}\bone_I>\bzero_I$ imply that
    \begin{eqnarray*}
        \kappa_{\mathbb{I}_d}=\bone_I^{\top}\Sigma_{I}^{-1}\bone_I^{\top}=\kappa_I.
    \end{eqnarray*}
    Further, we already know from the Schur decomposition \eqref{Schur}, which is valid independent of the choice of the set $S$, that $\kappa_{\mathbb{I}_d}\geq\kappa_S\geq \kappa_I$.
    Hence the only possibility is $\kappa_{\mathbb{I}_d}=\kappa_S=\kappa_I$. The second statement was already proven in \eqref{eq2}.
\end{proof}

The next proposition provides the tail asymptotics for the Gaussian survival copula using \citet[Theorem 1]{das:fasen:2023a}.
\begin{proposition}\label{prop:gcsurvcop}
    Let $C^{\Sigma}$ be a  Gaussian copula with positive definite correlation matrix $\Sigma$ and $S\subset \mathbb{I}_d$ with $|S|\ge 2$. Let $\kappa_S$, $I_S$, 
and  $h_{s}^S, s\in I_S$, be defined as in  \Cref{lem:quadprog}. Now, with $\bv_S=(v_s)_{s\in S}$ where $v_s\in (0,1), \forall s\in S$, we have as $u\downarrow 0$,
    \begin{align}
         \widehat{C}^{\Sigma}_S(u\bv_S) 
        &= (1+o(1))\Upsilon_S({2\pi})^{\frac{\kappa_S}2}u^{\kappa_S} (-2\log u)^{\frac{\kappa_S-|I_S|}{2}} \prod_{s\in I_S} v_s^{h_s^S}  
        \label{eq:limitgcsurvival} 
    \end{align}
        where $\Upsilon_S>0$ is a constant.
\end{proposition}
    \begin{proof}
      Since \eqref{eq:limitgcsurvival} is independent of the marginals of the  distribution, consider a random vector $\bZ\sim G$ in $\R^d$ with standard Pareto
marginals, i.e.,  $G_j(z)=\P(Z_j\le z)=1-z^{-1}, z\ge1$, $\forall j\in \mathbb{I}_d$, and dependence given by the Gaussian copula $C^{\Sigma}$. Using \citet[Theorem 1]{das:fasen:2023a} we have that for $\bz_S=(z_s)_{s\in S}$ with $z_s>0$ $\forall s\in S$, as $t\to\infty$,
 \begin{align} \label{eq:limitpgcwithtx}
        \P(Z_s > tz_s,\,\forall\, s\in S)
        &= (1+o(1))\Upsilon_S({2\pi})^{\frac{\kappa_S}2}t^{-\kappa_S} (2\log (t))^{\frac{\kappa_S-|I_S|}{2}} \prod_{s\in I_S} z_s^{-h_s^S}  
    \end{align}
    where $\Upsilon_S>0$ is a constant. Then 
\begin{align*}
    \widehat{C}^{\Sigma}_S(u\bv_S) & = \P(G_s(Z_s)>1-uv_s,\,\forall\, s\in S)\\
    & = \P(Z_s>u^{-1}v_s^{-1},\,\forall\, s\in S)   
\end{align*}
 and the result follows immediately from \eqref{eq:limitpgcwithtx}.  
    \end{proof}

\begin{lemma}\label{lem:premaiforGausscop}
      Let $C^{\Sigma}$ be a Gaussian copula with positive definite correlation matrix $\Sigma$. Then there exists a $\ell\in \mathbb{I}_d$  such that
      \begin{align}\label{eqa2}
          \lim_{u\downarrow 0} \frac{\widehat{C}^\Sigma(u,\ldots,u)}{\widehat{C}^\Sigma_{\mathbb{I}_d\setminus{\{\ell\}}}(u,\ldots,u)}
          = c\in \left(0,1\right]
      \end{align}
     if and only if $\Sigma^{-1}\bone \ngeq\bzero$.
\end{lemma}
\begin{proof} 
   $\Leftarrow$: \, Suppose $\Sigma^{-1}\bone \ngeq\bzero$. From \Cref{lem:aux1}(b) we already know that $I \not=\mathbb{I}_d$. Now  let $\ell\in\mathbb{I}_d\backslash I$.  For $S=\mathbb{I}_d\backslash\{\ell\}$ we have $I\subseteq S\subseteq \mathbb{I}_d$,  with $I=I_S$ and  $\kappa_{\mathbb{I}_d}=\kappa_S$ (cf. proof of \Cref{lem:aux1}). 
Now using \eqref{eq:limitgcsurvival} we have
\begin{align*}
     \lim_{u\downarrow 0} \frac{\widehat{C}^\Sigma(u,\ldots,u)}{\widehat{C}^\Sigma_{\mathbb{I}_d\setminus{\{\ell\}}}(u,\ldots,u)} 
    & = \frac{{\Upsilon}_{\mathbb{I}_d}}{\Upsilon_{\mathbb{I}_d\setminus\{\ell\}}} >0.
\end{align*}

$\Rightarrow$: \, Suppose there exists $\ell\in \mathbb{I}_d$ such that \eqref{eqa2} holds. We prove the statement by contradiction. By way of contradiction, assume
$\Sigma^{-1}\bone >\bzero$ holds. \Cref{lem:aux1} says that  for any set $S\subseteq \mathbb{I}_d$ with  $S\not= \mathbb{I}_d$ the inequality $\kappa_{\mathbb{I}_d}>\kappa_S$ holds. 
Again using \eqref{eq:limitgcsurvival} we have  with $\kappa^*:=\kappa_{\mathbb{I}_d}-\kappa_{\mathbb{I}_d\setminus\{\ell\}}$ and $d^*:=d-|I_{\mathbb{I}_d\setminus\{\ell\}}|$,
\begin{align*}
     \lim_{u\downarrow 0} \frac{\widehat{C}^\Sigma(u,\ldots,u)}{\widehat{C}^\Sigma_{\mathbb{I}_d\setminus{\{\ell\}}}(u,\ldots,u)} 
      & =  \lim_{u\downarrow 0} \frac{\Upsilon_{\mathbb{I}_d}}{\Upsilon_{\mathbb{I}_d\setminus\{\ell\}}} (\sqrt{2\pi}u)^{\kappa^*} 
      (-2\log u)^{\frac{\kappa^*-d^*}{2}} = 0
\end{align*}
which is a contradiction to \eqref{eqa2}.
\end{proof}

\begin{proof}[Proof of \Cref{thm:maiforGausscop}]
    The proof follows now from \Cref{lem:premaiforGausscop} by using an analogous argument as given in the proof of \Cref{prop:gcsurvcop}.
\end{proof}

\begin{proof}[Proof of \Cref{prop:Gausscoptailorder}]
    The proof directly follows from \Cref{prop:gcsurvcop} where a representation for $\ell_S$ is also provided.
\end{proof}

\section{Proofs of Section \ref{sec:kwiseasyind}} \label{Appendix B}

\begin{proof}[Proof of \Cref{Theorem 4.4}]
First, we define for some $\rho\in(-\frac{1}{\sqrt{k}},\frac{1}{\sqrt{k}})$  the $\R^{(k+1)\times (k+1)}$-valued  positive definite matrix 
\begin{eqnarray*}
	\Gamma_{\rho}:=\begin{bmatrix}
		\II_k\;\;\; & \rho \bone_k\\[1em]
           \rho \bone_k^{\top} & 1
\end{bmatrix}
\end{eqnarray*}
with inverse
\begin{eqnarray}
 	\Gamma_{\rho}^{-1}=\begin{bmatrix}
		\II_k+\frac{\rho^2}{1-k\rho^2}\bone_k\bone_k^{\top}\;\;\;\;\;  & \frac{-\rho}{1-k\rho^2} \bone_k\\[1em]
           \frac{-\rho}{1-k\rho^2} \bone_k^{\top} & \frac{1}{1-k\rho^2}
\end{bmatrix}.
\end{eqnarray}
Note that

\begin{eqnarray*}
	\Gamma_{\rho}^{-1}\bone_{k+1}
		=\left[\frac{1-\rho}{1-k\rho^2},\ldots,\frac{1-\rho}{1-k\rho^2},\frac{1-k\rho}{1-k\rho^2}\right]^\top.
\end{eqnarray*}
If we restrict $\rho\in[\frac{1}{k},\frac{1}{\sqrt{k}})$ then the first $k$ components of  
$\Gamma_{\rho}^{-1}\bone_{k+1}$ are positive and the last component is negative resulting in $\Gamma_{\rho}^{-1}\bone_{k+1} \ngtr \bzero_{k+1}$, and hence, due to \Cref{thm:maiforGausscop}, a Gaussian copula $C^{\Gamma_{\rho}}$ with correlation matrix $\Gamma_{\rho}$ is not mutually asymptotically independent and thus, not $(k+1)$-wise asymptotically independent.

 Now suppose that $\bX\in\R^{(k+1)\times(k+1)}$ is a random vector with Gaussian copula $C^{\Gamma_{\rho}}$ where $\rho$ is further restricted to $\rho\in(\frac{1}{k},\min(\frac{1}{k-1},\frac{1}{\sqrt{k}}))$. Consider a subset $S\subset\{1,\ldots,k+1\}$  with $\vert S\vert=j$ such that $j \in \{2,\ldots, k\}$.
 \begin{itemize}
     \item If $k+1\in S$, considering $k+1$ to be the final element of $S$, we have
     \begin{eqnarray*}
	[\Gamma_{\rho}]_S^{-1}\bone_{j}
		=\left[\frac{1-\rho}{1-(j-1)\rho^2},\ldots,\frac{1-\rho}{1-(j-1)\rho^2},\frac{1-(j-1)\rho}{1-(j-1)\rho^2}\right]^\top>\bzero_j.
\end{eqnarray*}
\item If $k+1\notin S$, then $[\Gamma_{\rho}]_S={\II}_j$ and hence 
\[[\Gamma_{\rho}]_S^{-1}\bone_{j} = \bone_{j}>\bzero_j. \]
 \end{itemize}
Thus \Cref{thm:maiforGausscop} implies then that $\bX_J$, for any $J \subseteq \{1, \ldots, k+1\}$ with $\vert J \vert \leq k$, is a mutually asymptotically independent random vector in $\R^J$. Finally, a conclusion of \Cref{Remark 4.3} is that $\bX$ is $k$-wise asymptotically independent in $\R^{(k+1)\times(k+1)}$, although it is not $(k+1)$-wise asymptotically independent. From 
\Cref{lem:quadprog} we know that $I_{\{1,\ldots,k+1\}}=\{1,\ldots,k\}=I_{\{1,\ldots,k\}}$, $\kappa_{\{1,\ldots,k+1\}}=\kappa_{\{1,\ldots,k\}}=k$, \revision{$h_i^{\{1,\ldots,k+1\}}=h_i^{\{1,\ldots,k\}}=1$ } for $i\in \{1,\ldots,k\}$ and finally,
from \Cref{prop:gcsurvcop} that
 \begin{eqnarray*}
        \lim_{u\downarrow 0}\frac{\widehat C^{\Gamma_{\rho}}_{\{1,\ldots,k+1\}}(u,\ldots,xu)}{\widehat C^{\Gamma_{\rho}}_{\{1,\ldots,k\}}(u,\ldots,u)}=1.
  \end{eqnarray*}
  Note that the constant $\Upsilon_S$ in \Cref{prop:gcsurvcop} is not specified in this paper, but it is given in \citet[Theorem 1]{das:fasen:2023a}, from which we obtain $\Upsilon_{\{1,\ldots,k+1\}} = \Upsilon_{\{1,\ldots,k\}}$.

After all, define the $(d\times d)$-dimensional correlation $\Sigma_{\rho}$ as a block diagonal matrix having in the first $(d-(k+1))\times (d-(k+1))$ block the identity matrix, {zeros in the two off-diagonal blocks,} and, in the last $(k+1)\times(k+1)$ block $\Gamma_{\rho}$ with $\rho\in(\frac{1}{k},\min(\frac{1}{k-1},\frac{1}{\sqrt{k}}))$, i.e., the random vector $\bZ^*=(Z_1^*,\ldots,Z_d^*)$ with Gaussian copula $C^{\Sigma_{\rho}}$ has the property that $Z_1^*,\ldots,Z_{d-(k+1)}^*$ are an independent sequence which is as well independent of the random vector $\bX=(Z_{d-k}^*,\ldots, Z_d^*)$ in $\R^{(k+1)\times(k+1)}$ with Gaussian copula $C^{\Gamma_{\rho}}$. Then by analogous arguments as above,  $\bZ^*$ is a $k$-wise asymptotically independent random vector in $\R^d$ although it is not $(k+1)$-wise asymptotically independent and
\begin{eqnarray*}
        \lim_{u\downarrow 0}\frac{\widehat C^{\Sigma_{\rho}}_{\{d-k,\ldots,d\}}(u,\ldots,xu)}{\widehat C^{\Sigma_{\rho}}_{\{d-k,\ldots,d-1\}}(u,\ldots,u)}=1.
  \end{eqnarray*}
The $d$-dimensional random vector $\bZ$ is finally a permutation of $\bZ^*$ with $
\bZ_{S_2}=\bZ^*_{\{d-k,\ldots,d\}}$, $
\bZ_{S_1}=\bZ^*_{\{d-k,\ldots,d-1\}}$ and $
\bZ_{\mathbb{I}_{d}\backslash S_2}=\bZ^*_{\{1,\ldots,d-k-1\}}$ and satisfies the requirements of the theorem.
\end{proof}

\section{Proofs of Section \ref{sec:dro}} \label{Appendix C}

\begin{proof}[Proof of \Cref{Theorem 5.1}]
    For ease of notation, we define $J^*:=J\cup\{1\}$. By definition,
    \begin{align}
        \CTP_{\gamma}(\bZ_{1|J}) & = \P\left(Z_1>\VaR_{\gamma}(Z_1) | Z_j>\VaR_{\gamma}(Z_j),\, \forall j\in J\right) \nonumber\\ & = \P (Z_1>{F}_1^{-1}(1-\gamma)| Z_j>{F}_j^{-1}(1-\gamma) ,\, \forall j\in J) \nonumber \\
        & = \frac{\widehat{C}_{J^*}(\gamma, \ldots, \gamma)}{\widehat{C}_{J}(\gamma, \ldots, \gamma)}, \label{eq:CTPlim}
    \end{align}
    which does not depend on the marginal distributions.
    \begin{enumerate}[(a)]
    \item Since $\mathcal{N}_k\subseteq \mathcal{P}_k$, and probabilities are non-negative, it is sufficient to show the statement for $\mathcal{P}_k$. But for any $F\in\mathcal{P}_k$, by definition of $k$-wise asymptotic independence and $\vert J^*\vert=\ell+1\leq k$ we have 
   $\lim_{\gamma\downarrow 0} \CTP_{\gamma}(\bZ_{1|J}) =0$,
   and thus (a) holds.
   
    \item If $d=2$, there is nothing else to prove.    
    Hence, now assume $d\ge 3$. Since $0\le \CTP_{\gamma}(\bZ_{1| J}) \le 1$, to show (b), it is sufficient to provide an example of $F\in \mathcal{N}_\ell\subseteq  \mathcal{N}_k\subseteq \mathcal{P}_k$ for $k\in\{2,\ldots,\ell\}$, such that for $\bZ\sim F$, we have 
   $\lim_{\gamma\downarrow 0} \CTP_{\gamma}(\bZ_{1|J}) =1$ . 
   To this end, we will choose $F$ with a Gaussian copula
   $C^{\Sigma}$ and positive-definite correlation matrix $\Sigma$ as identified in \Cref{Theorem 4.4}, such that $F$ exhibits $\ell$-wise asymptotic independence but not $(\ell+1)$-wise asymptotic independence and for any $x>0$,
  \begin{eqnarray*} \label{5.8}
        \lim_{\gamma\downarrow 0}\frac{\widehat C^{\Sigma}_{J^*}(x\gamma,\gamma,\ldots,\gamma)}{\widehat C^{\Sigma}_{J}(\gamma,\ldots,\gamma)}=1. 
  \end{eqnarray*}
    Hence, $F\in\mathcal{N}_{\ell}$ and by \eqref{eq:CTPlim} we have as well
 \begin{align*}
        \lim_{\gamma\downarrow 0}\CTP_{\gamma}(\bZ_{1|J}) & = \lim_{\gamma\downarrow 0}\frac{\widehat{C}_{J^*}^{\Sigma}(\gamma, \ldots, \gamma)}{\widehat{C}_{J}^{\Sigma}(\gamma, \ldots, \gamma)} =1,
    \end{align*}
    which we wanted to show.     
\end{enumerate} \vspace*{-0.85cm}
\end{proof}

\begin{proof}[Proof of \Cref{Lemma 5.3}]
By definition we have the relation $\mathcal{N}_k^*\subseteq \mathcal{N}_k \subseteq \mathcal{P}_k$. Since distributions in $\mathcal{N}_k^*$ have  a Pareto distributed margin in the first component, it remains to show that for any Gaussian copula $C^{\Sigma}$, where 
$\Sigma$ is a positive definite correlation matrix,
\begin{align} \label{AAa}
\sup_{\gamma\in(0,x^{-1}]}\frac{\widehat C_S^{\Sigma}(x\gamma,\gamma,\ldots,\gamma)}{\widehat C_S^{\Sigma}(\gamma,\gamma,\ldots,\gamma)}<\infty
\end{align}
for all  $S\subseteq \mathbb{I}_d,$ and for all $ x\geq 1$. However, a conclusion from  \Cref{prop:gcsurvcop}
is that for any $S\subseteq \mathbb{I}_d,$ there exists a   constant $h_1^S\geq 0$ (where $h_1^S=0$ if $1\notin I_S$) so that for any $x>0$,
$$ 
    \lim_{\gamma\downarrow 0}\frac{\widehat C_S^{\Sigma}(x\gamma,\gamma,\ldots,\gamma)}{\widehat C_S^{\Sigma}(\gamma,\gamma,\ldots,\gamma)}=x^{h_1^S}
$$
implying \eqref{AAa}.
\end{proof}

\begin{proof}[Proof of \Cref{Theorem 5.4}]
First, note that 
\begin{eqnarray*}
    \text{CoVaR}_{\gamma_1,\gamma_2}
    (\bZ_{1|J})
    &=&\inf\{z\in\R_+:
    \P(Z_1>z\vert Z_j>\VaR_{\gamma_2}(Z_j),\,\forall j\in J)\leq \gamma_1\}\\
    &=&\VaR_{\gamma_2}(Z_1)\inf\{z\in\R_+:
    \P(Z_1>z\VaR_{\gamma_2}(Z_1)\vert Z_j>\VaR_{\gamma_2}(Z_j),\,\forall j\in J)\leq \gamma_1\}.
\end{eqnarray*}
Suppose  $Z_1$ is Pareto$(\alpha)$-distributed, $\alpha>0$. 
Then the previous equation  reduces to
\begin{eqnarray} \label{m1}
    \text{CoVaR}_{\gamma_1,\gamma_2}
    (\bZ_{1|J})
    =\VaR_{\gamma_2}(Z_1)\inf\left\{z\in\R_+:
    \frac{\widehat C_{J^*}(z^{-\frac{1}{\alpha}} \gamma_2,\gamma_2,\ldots,\gamma_2)}{\widehat C_{J}(\gamma_2,\ldots,\gamma_2)}\leq \gamma_1\right\},
\end{eqnarray}
where $J^*=J\cup\{1\}$.
\begin{enumerate}[(a)]
    \item Suppose $F\in\mathcal{P}_k^*$ and $k\in\{\ell+1,\ldots,d\}$. Let $\epsilon\in(0,\gamma_1)$ and 
    \begin{align}\label{eq:K}
    K&:=\sup_{\gamma\in(0,\epsilon^{1/\alpha}]}\frac{\widehat C_{J^*}(\epsilon^{-1/\alpha}\gamma,\gamma,\ldots,\gamma)}{\widehat C_{J^*}(\gamma,\gamma,\ldots,\gamma)},
    \end{align}
     which is finite for $F\in\mathcal{P}_k^*$ by the definition of $\mathcal{P}_k^*$.
    Furthermore, $F\in\mathcal{P}_k^* {\subseteq} \mathcal{P}_k$ implies that there exists a $\gamma_0(\epsilon)\in(0,\gamma_1)$ such that
    \begin{eqnarray}
        \frac{\widehat C_{J^*}(\gamma,\gamma,\ldots,\gamma)}{\widehat C_{J}(\gamma,\gamma,\ldots,\gamma)}\leq\frac{\epsilon}{K}, \quad \forall\, \gamma\in(0,\gamma_0(\epsilon)). \label{eq:epbyK}
    \end{eqnarray}
    Therefore from \eqref{eq:K} and \eqref{eq:epbyK},  for all $0< \gamma_2<\min(\epsilon^{1/\alpha},\gamma_0(\epsilon))$ we have
    \begin{eqnarray*}
        \frac{\widehat C_{J^*}(\epsilon^{-\frac{1}{\alpha}}{\gamma_2},\gamma_2,\ldots,\gamma_2)}{\widehat C_{J}(\gamma_2,\ldots,\gamma_2)}=\frac{\widehat C_{J^*}(\epsilon^{-\frac{1}{\alpha}}{\gamma_2},\gamma_2,\ldots,\gamma_2)}{\widehat C_{J^*}(\gamma_2,\ldots,\gamma_2)}\frac{\widehat C_{J^*}(\gamma_2,\ldots,\gamma_2)}{\widehat C_{J}(\gamma_2,\ldots,\gamma_2)}\leq K\cdot \frac{\epsilon}{K}<\gamma_1,
    \end{eqnarray*}
    and finally, using \eqref{m1}, we get
    \begin{eqnarray*}        
     \frac{\text{CoVaR}_{\gamma_1,\gamma_2}(\bZ_{1|J})}{\VaR_{\gamma_2}(Z_1)}\leq \epsilon.
    \end{eqnarray*}
    Since $\epsilon\in(0,\gamma_1)$ is arbitrary, this results in
    $$\lim_{\gamma_2\downarrow 0} 
     \frac{\text{CoVaR}_{\gamma_1,\gamma_2}(\bZ_{1|J})}{\VaR_{\gamma_2}(Z_1)}=0.$$
     Finally, from \Cref{Lemma 5.3} we already know that $\mathcal{N}_k^*\subseteq \mathcal{P}_k^*$, thus the result is true for $\mathcal{N}_k^*$ as well.
     \item We will construct a $\bZ\sim F\in\mathcal{N}_{\ell}^*$,  so that
\begin{eqnarray*}
    \lim_{\gamma_2\downarrow 0}\frac{\text{CoVaR}_{\gamma_1,\gamma_2}(\bZ_{1|J})}{\VaR_{\gamma_2}(Z_1)}=\infty,
\end{eqnarray*}
which shows the statement.

     To this end, we will choose $\bZ\sim F$ which has a Gaussian copula
   $C^{\Sigma}$ with positive-definite correlation matrix $\Sigma$ as in \Cref{Theorem 4.4}, such that $F$ exhibits $\ell$-wise asymptotic independence but not $(\ell+1)$-wise asymptotic independence and for any $x>0$,
  \begin{eqnarray} \label{5.8}
        \lim_{u\downarrow 0}\frac{\widehat C^{\Sigma}_{J^*}(xu,u,\ldots,u)}{\widehat C^{\Sigma}_{J}(u,\ldots,u)}=1. 
  \end{eqnarray}
  Additionally, suppose that the margin $F_1$ is Pareto($\alpha$)-distributed. 
  Then $F\in\mathcal{N}_{\ell}^*\subseteq \mathcal{N}_{k}^*\subseteq \mathcal{P}_{k}^*$ for $k\in\{2,\ldots,\ell\}$.
 Due to \eqref{5.8}, for any $M>0$ there exists an $\gamma_0(M)\in(0,1)$ such that
\begin{eqnarray*}
 \frac{\widehat C^{\Sigma}_{J^*}(M^{-\frac{1}{\alpha}}\gamma_2,\gamma_2,\ldots,\gamma_2)}{\widehat C^{\Sigma}_{J}(\gamma_2,\ldots,\gamma_2)}>\frac{\gamma_1+1}{2}, \quad \forall\, \gamma_2\in(0,\gamma_0(M)).
\end{eqnarray*}
From this, we get that $\forall\, \gamma_2\in(0,\gamma_0(M))$,
 \begin{eqnarray*}
    \frac{\text{CoVaR}_{\gamma_1,\gamma_2}(\bZ_{1|J})}{\VaR_{\gamma_2}(Z_1)}
    &=&\inf\left\{z\in\R_+:
    \frac{\widehat C_{J^*}(z^{-\frac{1}{\alpha}}\gamma_2,\gamma_2,\ldots,\gamma_2)}{\widehat C_{J}(\gamma_2,\ldots,\gamma_2)}\leq \gamma_1\right\}\geq M,
\end{eqnarray*}
implying 
 \begin{eqnarray*}
    \liminf_{\gamma_2\downarrow 0}\frac{\text{CoVaR}_{\gamma_1,\gamma_2}(\bZ_{1|J})}{\VaR_{\gamma_2}(Z_1)}
    \geq M.
\end{eqnarray*}
Since $M>0$ is arbitrary, we have
\begin{eqnarray*}
    \lim_{\gamma_2\downarrow 0}\frac{\text{CoVaR}_{\gamma_1,\gamma_2}(\bZ_{1|J})}{\VaR_{\gamma_2}(Z_1)}=\infty,
\end{eqnarray*}
exhibiting the desired property for our chosen $F$ and, hence, proving the result.
\end{enumerate} \vspace*{-0.7cm}
\end{proof}

\end{document}